\documentclass[12pt]{amsart}
\usepackage{a4wide,enumerate,enumitem,amsmath,color}
\allowdisplaybreaks

\let\pa\partial
\let\na\nabla
\let\eps\varepsilon
\newcommand{\N}{{\mathbb N}}
\newcommand{\R}{{\mathbb R}}
\newcommand{\diver}{\operatorname{div}}

\newtheorem{theorem}{Theorem}
\newtheorem{lemma}[theorem]{Lemma}

\newtheorem{remark}[theorem]{Remark}

\newtheorem{definition}{Definition}

\begin{document}

\title[Weak-strong uniqueness]{Weak-strong uniqueness of renormalized solutions
to reaction-cross-diffusion systems}

\author[X. Chen]{Xiuqing Chen}
\address{School of Sciences, Beijing University of Posts and Telecommunications,
Beijing 100876, China}
\email{buptxchen@yahoo.com}

\author[A. J\"ungel]{Ansgar J\"ungel}
\address{Institute for Analysis and Scientific Computing, Vienna University of
	Technology, Wiedner Hauptstra\ss e 8--10, 1040 Wien, Austria}
\email{juengel@tuwien.ac.at}

\date{\today}

\thanks{The first author acknowledges support from the National Natural Science
Foundation of China (NSFC), grant 11471050, and from the China Scholarship Council
(CSC), file no.\ 201706475001, who financed his stay in Vienna.
The second author acknowledges partial support from
the Austrian Science Fund (FWF), grants F65, I3401, P27352, P30000, and W1245}

\begin{abstract}
The weak-strong uniqueness for solutions to reaction-cross-diffusion systems
in a bounded domain with no-flux boundary conditions is proved. The system
generalizes the Shigesada-Kawasaki-Teramoto population model to an arbitrary
number of species. The diffusion matrix is neither symmetric nor positive definite,
but the system possesses a formal gradient-flow or entropy structure.
No growth conditions on the source terms are imposed.
It is shown that any renormalized solution coincides with a strong solution
with the same initial data, as long as the strong solution exists.
The proof is based on the evolution of the relative entropy modified by
suitable cutoff functions.
\end{abstract}



\maketitle


\section{Introduction}

This paper is a continuation of our work \cite{ChJu17}, in which we proved the
global existence of renormalized solutions to a class of reaction-cross-diffusion
systems describing the evolution of population species. The reaction part does
not obey any growth condition which makes it necessary to use the concept
of renormalized solutions like in \cite{Fis15}. The uniqueness of weak solutions
to cross-diffusion systems is a very delicate topic, and there are very few results
only for special problems; we refer to \cite{ChJu18} and references therein.
In this work, we show a weak-strong uniqueness result for the population
cross-diffusion system. This means that any
renormalized solution coincides with a strong solution emanating from the same
initial data as long as the latter exists. This paper generalizes the weak-strong
uniqueness result of Fischer \cite{Fis17} for {\em semilinear}
reaction-diffusion systems to {\em quasilinear} reaction-cross-diffusion problems.

More specifically, we consider the evolution of $n$ population species
with densities $u_i=u_i(x,t)$, $i=1,\ldots,n$, whose evolution
is governed by the equations
\begin{equation}\label{1.eq}
  \pa_t u_i - \diver\bigg(\sum_{j=1}^n A_{ij}(u)\na u_j - u_ib_i\bigg) = f_i(u)
	\quad\mbox{in }\Omega,\ i=1,\ldots,n,
\end{equation}
where $A_{ij}(u)$ are the density-dependent diffusion coefficients,
$u=(u_1,\ldots,u_n)$ is the density vector, $b_i\in\R^d$ is
a given vector which describes the environmental potential acting on the $i$th species,
$f_i(u)$ is a reaction term describing the population growth dynamics,
and $\Omega\subset\R^d$ ($d\ge 1$) is a bounded domain.
We impose no-flux boundary and initial conditions,
\begin{equation}\label{1.bic}
  \bigg(\sum_{j=1}^n A_{ij}(u)\na u_j - u_ib_i\bigg)\cdot\nu = 0\quad\mbox{on }\pa\Omega,
	\quad u_i(\cdot,0)=u_i^0\quad\mbox{in }\Omega,\ i=1,\ldots,n,
\end{equation}
where $\nu$ is the exterior unit normal vector on $\pa\Omega$.
The diffusion coefficients are given by
\begin{equation}\label{1.A}
  A_{ij}(u) = \delta_{ij}\bigg(a_{i0} + \sum_{k=1}^n a_{ik}u_k\bigg) + a_{ij}u_i,
	\quad i,j=1,\ldots,n,
\end{equation}
where $a_{i0}\ge 0$, $a_{ij}\ge 0$ for $i,j=1,\ldots,n$, and $\delta_{ij}$ is the
Kronecker delta. Observe that the diffusion matrix is generally neither symmetric
nor positive definite, which constitutes a major difficulty in the analysis of
the diffusion system. This problem is overcome by exploiting its entropy structure,
which is explained below.


\subsection{State of the art}

System \eqref{1.eq}-\eqref{1.A} has been suggested by Shigesada, Kawasaki, and
Teramoto for $n=2$ species to describe the segregation of populations \cite{SKT79}.
The equations (for any $n\ge 2$) were derived from a random-walk on a lattice
in the diffusion limit \cite{ZaJu17}.
The global existence of nonnegative weak solutions to
\eqref{1.eq}-\eqref{1.A} for two species was proved in \cite{ChJu04} for any
coefficients $a_{ij}>0$. This result was generalized to an arbitrary number of
species in \cite{CDJ18}, under a growth condition on the source terms. This
condition could be replaced by a weaker entropy-dissipation assumption, yielding
the global existence of renormalized solutions \cite{ChJu17}.

The concept of renormalized solutions has been introduced by DiPerna and Lions
for transport and Boltzmann equations \cite{DiLi88,DiLi89,DiLi89a}.
The idea is to replace the solution $u$ by a
nonlinear function $\xi(u)$ with compact support. This concept
was applied also to elliptic and parabolic problems (e.g.\ \cite{AIMT08,DMDP99})
and diffusion systems (e.g.\ \cite{DFPV07,Fis15}).

Weak-strong uniqueness was established by Leray \cite{Ler34} for incompressible
Navier-Stokes equations and by Dafermos \cite{Daf10} for conservation laws;
see the review by Wiedemann \cite{Wie17} for more details.
Later this concept has been applied
to other fluid models, including measure-valued solutions \cite{FeNo12,GSW15};
to magneto-viscoelastic flow equations \cite{ScZa17}; and
to gradient flows based on optimal transport \cite{BrDu17}.
As far as we know, there are very few works on the weak-strong uniqueness
involving renormalized solutions. An example is the paper \cite{FeTu17},
where the weak-strong uniqueness for renormalized relaxed Lagrangian solutions
to semi-geostrophic equations was shown, and the already mentioned work \cite{Fis17}
by Fischer on the weak-strong uniqueness for renormalized solutions to
reaction-diffusion systems.

The question of uniqueness of weak solutions to parabolic diffusion systems
is extremely delicate. One of the first results is due to Alt and Luckhaus
\cite{AlLu83} for linear elliptic operators. Pham and Temam \cite{PhTe17} proved
a uniqueness result for the population system \eqref{1.eq}-\eqref{1.A}, but only
for two species and assuming a positive definite diffusion matrix. Finally,
Gajewski's uniqueness method was applied to a simplified
volume-filling cross-diffusion system
in \cite{ZaJu17}. Up to our knowledge, there does not exist any uniqueness result
for generalized solutions to the population system \eqref{1.eq}-\eqref{1.A}
without the assumptions imposed in \cite{PhTe17}.


\subsection{Key ideas}\label{sec.ideas}

The analysis of \eqref{1.eq}-\eqref{1.A} is based on its entropy structure.
This means that under some conditions, there exists a convex
Lyapunov functional, which is called an entropy and which yields gradient estimates.
The entropy gives rise to a transformation
to entropy variables that makes the transformed diffusion matrix positive semidefinite,
thus reveiling the parabolic structure of the evolution system.
For this result, we need two assumptions. The first one are entropy-dissipating
source terms, which means that there exist numbers $\pi_1,\ldots,\pi_n>0$ and
$\lambda_1,\ldots,\lambda_n\in\R$ such that
\begin{equation}\label{1.f}
  \sum_{i=1}^n\pi_i f_i(u)(\log u_i+\lambda_i)\le 0 \quad\mbox{for all }u\in(0,\infty)^n.
\end{equation}
This condition implies the quasi-positivity of $f_i$ which is necessary to conclude
nonnegative solutions to \eqref{1.eq}.
Note that we do not impose any growth restriction on the reaction terms,
modeling possibly fast growing populations.

Condition \eqref{1.f} ensures that the entropy density
\begin{equation}\label{1.h}
  h(u) = \sum_{i=1}^n \pi_i h_i(u_i), \quad
	h_i(s) = s(\log s-1+\lambda_i) + e^{-\lambda_i},
\end{equation}
is a Lyapunov functional for the pure reaction system $\pa_t u_i=f_i(u)$ if
$\pi_i=1$ for all $i=1,\ldots,n$. When the diffusion terms are present, a second
assumption is needed, namely either the weak cross-diffusion condition
\begin{equation}\label{1.wcd}
  \eta := \min_{i=1,\ldots,n}\bigg(a_{ii} - \frac14\sum_{j=1}^n
	\big(\sqrt{a_{ij}}-\sqrt{a_{ji}}\big)^2\bigg) > 0,
\end{equation}
or the detailed-balance condition
\begin{equation}\label{1.dbc}
  \pi_ia_{ij} = \pi_ja_{ji} \quad\mbox{for all }i,j=1,\ldots,n,\ i\neq j.
\end{equation}
In the former case, we may choose $\pi_i=1$. For an interpretation of the
detailed-balance condition, we refer to \cite{CDJ18}.

Under conditions \eqref{1.f} and either \eqref{1.wcd} or \eqref{1.dbc},
the matrix product $A(u)h''(u)^{-1}$ is positive semidefinite
(here, $h''(u)$ denotes the Hessian of $h(u)$), i.e.\ for any $z\in\R^n$,
\begin{equation}\label{1.posdef}
  z:A(u)h''(u)^{-1}z = \sum_{i,j=1}^n A_{ij}(u)u_jz_iz_j
	\ge \alpha_0\sum_{i=1}^n u_iz_i^2 + 2\eta_0\sum_{i=1}^n u_i^2z_i^2,
\end{equation}
for some constants $\alpha_0$, $\eta_0>0$; see Lemma \ref{lem.psd} below.
As a consequence,
the entropy $\int_\Omega h(u)dx$ is a Lyapunov functional along solutions
to \eqref{1.eq}-\eqref{1.A}, and we obtain the so-called entropy inequality
\begin{equation}\label{1.ent}
  \frac{d}{dt}\int_\Omega h(u)dx
	+ C\int_\Omega\sum_{i=1}^n\big(|\na\sqrt{u_i}|^2 + |\na u_i|^2\big)dx
  \le 0,
\end{equation}
where the constant $C>0$ depends on $\pi_i$ and $a_{ij}$.
Clearly, these assumptions are also needed for our uniqueness result. In fact,
we need an additional condition on the reaction terms detailed in hypothesis
(H2) below.

As in \cite{Fis17}, the key idea of the uniqueness proof is the use of the
relative entropy,
\begin{align*}
  H(u|v) &= \int_\Omega\big(h(u)-h(v)-h'(v)\cdot(u-v)\big)dx \\
	&= \sum_{i=1}^n\int_\Omega\big(u_i(\log u_i-1) - u_i\log v_i + v_i\big)dx,
\end{align*}
which can be seen as a generalized distance between a renormalized solution
$u$ and a strong solution $v$. There is a relation between Gajewski's semimetric
and the relative entropy; see the discussion in \cite[Remark 4]{ChJu18}.
To simplify the following formal arguments (which are made rigourous
in section \ref{sec.proof}), we set $b_i=0$, $\lambda_i=0$, and $\pi_i=1$.
A computation shows that
\begin{align}
  \frac{dH}{dt}(u|v)
	&= -\sum_{i,j=1}^n\int_\Omega A_{ij}(u)u_j\na\log\frac{u_i}{v_i}
	\cdot\na\log\frac{u_j}{v_j}dx \nonumber \\
	&\phantom{xx}{}- \sum_{i,j=1}^n \int_\Omega\bigg(A_{ij}(u)\frac{u_j}{v_j}
	- A_{ij}(v)\frac{u_i}{v_i}\bigg)
	\na v_j\cdot\na\log\frac{u_i}{v_i}dx \label{1.dHdt} \\
	&\phantom{xx}{}+ \sum_{i=1}^n\int_\Omega \bigg(f_i(u)
	\log\frac{u_i}{v_i}+f_i(v)\left(1-\frac{u_i}{v_i}\right)\bigg)dx
	=: G_1 + G_2 + G_3. \nonumber
\end{align}
The second term $G_2$ is a result of the strong coupling and does not
appear in reaction-diffusion systems with diagonal and constant
diffusion matrix as in \cite{Fis17}.
The positive semidefiniteness property \eqref{1.posdef} shows that the
first term $G_1$ can be estimated from below,
\begin{equation}\label{1.G1}
  G_1 \le -2\eta_0\sum_{i=1}^n\int_\Omega u_i^2\bigg|\na\log\frac{u_i}{v_i}\bigg|^2 dx.
\end{equation}
Using the special structure \eqref{1.A} of the diffusion matrix, the second term
$G_2$ can be reformulated and estimated as
\begin{align*}
  G_2 &= -\sum_{i,j=1}^n a_{ij} u_i(u_j-v_j)
	\na\log(v_iv_j)\cdot\na\log\frac{u_i}{v_i}dx \\
	&\le C(v)\sum_{i,j=1}^n\int_\Omega |u_j-v_j|u_i\bigg|\na\log\frac{u_i}{v_i}\bigg|dx \\
	&\le \eta_0\sum_{i=1}^n\int_\Omega u_i^2\bigg|\na\log\frac{u_i}{v_i}\bigg|^2 dx
	+ C(v)\sum_{i=1}^n\int_\Omega|u_i-v_i|^2 dx.
\end{align*}
The first term on the right-hand side is absorbed by the right-hand
side of \eqref{1.G1}. The convexity of $h(u)$ shows that the relative entropy
is bounded from below by $\sum_{i=1}^n|u_i-v_i|^2$ (up to some constant),
provided that $u$ is bounded. In that situation, we infer that
$$
  \frac{dH}{dt}(u|v) \le C(v)H(u|v) + G_3, \quad t>0.
$$
Since we cannot prove the boundedness of $u$, we cannot use the relative entropy
directly. We need to construct a modified entropy with cutoff for $u_i$, such that the
previous arguments can be made rigorous. Note that this difficulty does not appear
when the diffusion matrix is diagonal and constant, as in \cite{Fis17}.
Indeed, then the term $G_2$ does not appear, and the only difficulty is to estimate
the remaining term $G_3$.

The idea of Fischer \cite{Fis17} to estimate $G_3$
is to introduce the relative entropy with cutoff for $v_i$,
$$
  \widetilde H_K^L(u|v) = \sum_{i=1}^n\int_\Omega\big(u_i(\log u_i+\lambda_i-1)
	-\widetilde\varphi_K^L(u)u_i(\log v_i+\lambda_i) + v_i\big)dx,
$$
where $K>3$, $L>0$ and $\widetilde\varphi_K^L$ is a cutoff function which
equals one if $\sum_{k=1}^n u_k\le L$ and vanishes if
$\sum_{k=1}^n u_k > (L+e)^{K}$,
$$
  \widetilde\varphi_K^L(u) = \varphi\bigg(\frac{\log(\sum_{k=1}^n u_k+e)-\log(L+e)}{
	(K-1)\log (L+e)}\bigg),
$$
$e=\exp(1)$ is the Euler number,
and $\varphi$ is a smooth cutoff such that $\varphi(s)=1$ if $s\le 0$
and $\varphi(s)=0$ if $s\ge 1$.
The cutoff allows for the control of $\widetilde\varphi_K^L(u)f_i(u)\log(1/v_i)$,
which appears in $G_3$ using $\widetilde H_K^L(u|v)$ instead of $H(u|v)$.

Unfortunately, this cutoff is not sufficient in the situation at hand, because of
the strong coupling in $G_1$ and $G_2$. Compared to \cite{Fis17},
we need two refinements. First, we introduce an additional cutoff:
\begin{equation}\label{1.enteps}
\begin{aligned}
  H_{K,\eps}^{M,L}(u|v) &= \int_\Omega\sum_{i=1}^n\Big(\varphi_K^M(u+\eps I)
	(u_i+\eps)\big(\log(u_i+\eps)+\lambda_i-1\big) \\
	&\phantom{xx}{}- \varphi_K^L(u+\eps I)(u_i+\eps)(\log u_i+\lambda_i) + v_i\Big)dx,
\end{aligned}
\end{equation}
where $M>L$, $\eps>0$, and $I=(1,\ldots,1)\in\R^n$. The parameter $\eps$ is
needed to control terms like $\log(u_i+\eps)$ when $u_i=0$.
Second, the cutoff function involves the double logarithm:
$$
  \varphi_K^L(u) := \varphi\bigg(\frac{\log\log(\sum_{k=1}^n u_k+e)
	- \log\log(L+e)}{\log(K+1)}\bigg).
$$
The additional logarithm slightly improves the estimates. Indeed,
$|\pa_j\widetilde\varphi_K^L(u)|$ is bounded by $C/[K(\sum_{k=1}^n u_k+e)]$,
while
$$
  |\pa_j\varphi_K^L(u)| \le \frac{C}{\log(K+1)(\sum_{k=1}^n u_k+e)
	\log(\sum_{k=1}^n u_k+e)}
$$
for some constant $C>0$.

These refinements allow us to estimate not only $G_1$ and $G_2$ but also $G_3$.
Then we can pass to the
limits $\eps\to 0$ and $M\to\infty$, yielding, for sufficiently large $K>0$,
$$
  \frac{dH_K^L}{dt}(u|v) \le C(K,L)H_K^L(u|v), \quad t>0,
$$
where $H_K^L(u|v):= \sum_{i=1}^n\int_\Omega\big(u_i(\log u_i+\lambda_i-1)
	-\varphi_K^L(u)u_i(\log v_i+\lambda_i) + v_i\big)dx$.
When $u$ and $v$ have the same initial data, we conclude for sufficiently
large $L>0$ that $H_K^L(u(t)|v(t))=0$
for all $t>0$ and hence, by Lemma \ref{lem.fis} below, $u(t)=u(t)$ for $t>0$.


\subsection{Main results}

First, we specify our notion of renormalized solution.

\begin{definition}\label{def.renorm}
We call $u=(u_1,\ldots,u_n)$ a {\em renormalized solution} to
\eqref{1.eq}-\eqref{1.A} if for all $T>0$, $u_i\in L^2(0,T;H^1(\Omega))$
or $\sqrt{u_i}\in L^2(0,T;H^1(\Omega))$,
and for any $\xi\in C^\infty([0,\infty)^n)$ satisfying
$\xi'\in C_0^\infty([0,\infty)^n;\R^n)$ and
$\phi\in C_0^\infty(\overline\Omega\times[0,T))$, it holds that
\begin{align}
  -\int_0^T\int_\Omega&\xi(u)\pa_t\phi dxdt
	- \int_\Omega\xi(u^0)\phi(x,0)dx \nonumber \\
  &= -\sum_{i,k=1}^n\int_0^T\int_\Omega\phi\pa_i\pa_k\xi(u)\bigg(
	\sum_{j=1}^nA_{ij}(u)\na u_j - u_ib_i\bigg)\cdot\na u_kdxdt \nonumber \\
  &\phantom{xx}{}- \sum_{i=1}^n\int_0^T\int_\Omega\pa_i\xi(u)\bigg(
	\sum_{j=1}^nA_{ij}(u)\na u_j - u_ib_i\bigg)\cdot\na \phi dxdt \nonumber \\
  &\phantom{xx}{}+ \sum_{i=1}^n\int_0^T\int_\Omega\phi\pa_i\xi(u)f_i(u) dxdt.
	\label{1.renorm}
\end{align}
\end{definition}

We impose the following hypotheses.

\renewcommand{\labelenumi}{(H\theenumi)}
\begin{enumerate}
\item[\rm (H1)] Drift term: $b=(b_1,\ldots,b_n)\in L^\infty(0,T;L^\infty(\Omega;
\R^{n\times d}))$
for $i=1,\ldots,n$.
\item[\rm (H2)] Reaction terms: (i) $f=(f_1,\ldots,f_n):[0,\infty)^n\to\R^n$ is
locally Lipschitz continuous; (ii) there exist numbers $\pi_1,\ldots,\pi_n>0$
and $\lambda_1,\ldots,\lambda_n>0$ such that
\begin{equation*}
  \sum_{i=1}^n\pi_i f_i(u)(\log u_i+\lambda_i)\le 0  \quad\mbox{for all }u\in(0,\infty)^n;
\end{equation*}
(iii) 
there exists $M_0\in\N$ such that for all $u\in[0,\infty)^n$ with
$\sum_{i=1}^n u_i\ge M_0$ it holds that $\sum_{i=1}^n f_i(u)\ge 0$.
\item[\rm (H3)] Initial data: $u^0=(u_1^0,\ldots,u_n^0)\in L^\infty(\Omega;\R^n)$
such that $\inf_\Omega u_i^0>0$ for $i=1,\ldots,n$.
\item[\rm (H4)] Diffusion coefficients: $a_{i0}>0$, $a_{ii}>0$ for $i=1,\ldots,n$
and either the weak cross-diffusion condition \eqref{1.wcd} holds and
$\pi_i=1$ for $i=1,\ldots,n$, or the detailed-balance condition \eqref{1.dbc} holds.
\end{enumerate}

\begin{remark}\label{rem}\rm
Under hypotheses (H1), (H2.i)-(H2.ii), (H3)-(H4), there exists a renormalized 
solution to \eqref{1.eq}-\eqref{1.A} satisfying $u_i\ge 0$ in $\Omega\times(0,T)$ and
$\int_\Omega h(u(t))dx<\infty$ for $t\in(0,T)$, and hence 
$u_i\in L^\infty(0,T;L^1(\Omega))$; see \cite{ChJu17}. If $a_{i0}>0$
and $a_{ii}>0$ for $i=1,\ldots,n$ then both functions $u_i$ and
$\sqrt{u_i}$ are in $L^2(0,T;H^1(\Omega))$. 
\qed
\end{remark}

\begin{remark}\rm
We discuss the assumptions.
Hypotheses (H1) and (H3) are rather natural. Condition (H2.ii) with
$\lambda_i=0$ was also imposed in \cite{DFPV07}, and we already mentioned
that it allows for the proof of the nonnegativity of the densities.
Condition (H2.iii) on the
positivity of $\sum_{i=1}^n f_i(u)$ may be surprising at first sight.
It means that in the absence of diffusion effects and for large total population,
the total population is still increasing. One would expect that an overcrowding
effect will lead to a decrease of the total population, thus requiring
$\sum_{i=1}^n f_i(u)\le 0$. However, in this situation, there is an upper bound
for the reaction terms and we can apply standard methods.
The situation becomes difficult when the total population
is not limited. This makes a priori estimate impossible (and makes necessary
the renormalization). An alternative condition is
$|\sum_{i=1}^n f_i(u)|\le C(1+|u|^p)$ for all $u\in[0,\infty)^n$ and $p=2+2/d$;
see Remark \ref{rem.alt}.
Finally, hypothesis (H4) is needed in the global existence
analysis to show that system \eqref{1.eq} has a certain parabolic structure; see
Lemma \ref{lem.psd} below.
\qed
\end{remark}

The main result of this paper reads as follows.

\begin{theorem}[Weak-strong uniqueness]\label{thm.ws}
Let (H1)-(H4) hold. Suppose that $u$ is a renormalized solution to
\eqref{1.eq}-\eqref{1.A} and $v$ is a ``strong'' solution to
\eqref{1.eq}-\eqref{1.A} on some time interval $[0,T^*)$ with $T^*\leq T$, in the
following sense: There exist $C>c>0$ such that
\begin{align}
  & c\le v_i(x,t)\le C\quad\mbox{for }(x,t)\in\Omega\times[0,T^*), \label{1.v1} \\
  &\|\pa_t v_i\|_{L^\infty(\Omega\times[0,T^*))}
	+ \|\na v_i\|_{L^\infty(\Omega\times[0,T^*))}\le C, \label{1.v2}
\end{align}
and for any $s\in(0,T^*)$, $\phi\in C^\infty(\overline\Omega\times[0,s])$,
and $i=1,\ldots,n$,
\begin{equation}\label{1.strong}
  \int_0^s\int_\Omega \phi\pa_t v_i dxdt
	= -\int_0^s\int_\Omega\bigg(\sum_{j=1}^nA_{ij}(v)\na v_j - v_ib_i\bigg)
	\cdot\na \phi dxdt + \int_0^s\int_\Omega\phi f_i(v) dxdt.
\end{equation}
Then $u(x,s)=v(x,s)$ for $x\in\Omega$, $s\in(0,T^*)$.
\end{theorem}

The population model \eqref{1.eq}-\eqref{1.A} can be derived from a random-walk
on-lattice model with transition rates that depend linearly on the densities
\cite{ZaJu17}. When the dependence is nonlinear (e.g.\ power functions),
we obtain population models
with coefficients $A_{ij}(u)$ that depend nonlinearly on $u_k$.
These models were analyzed in, e.g., \cite{DLMT15,ZaJu17}. However, it is unclear
to what extent the weak-strong uniqueness result can be extended to this case,
since the entropy density becomes a power function, and the construction of
suitable cutoff functions is an open problem.

As explained in section \ref{sec.ideas}, the proof of the theorem is highly
technical, involving two approximation levels with parameters
$\eps>0$, $M>0$, and $K>0$. The idea is to choose renormalizations
$\xi(u)$ involving $\varphi_K^L(u)$ and $\varphi_K^M(u)$ in
\eqref{1.renorm}, respectively, and to estimate all occuring terms, leading
to lengthy estimations. We summarize some
auxiliary results in section \ref{sec.aux} and present the proof of
Theorem \ref{thm.ws} in section \ref{sec.proof}.


\section{Some auxiliary results}\label{sec.aux}

As explained in the introduction, the matrix $A(u)h''(u)^{-1}$ is positive
semidefinite under hypothesis (H4). We recall the precise result.

\begin{lemma}\label{lem.psd}
Let hypothesis (H4) hold. Then for all $z\in\R^n$,
$$
  z:A(u)h''(u)^{-1}z
	= \sum_{i,j=1}^n A_{ij}(u)u_jz_iz_j
	\ge \alpha_0\sum_{i=1}^n u_iz_i^2
	+ 2\eta_0\sum_{i=1}^n u_i^2z_i^2,
$$
where the coefficients of $A(u)$ are given in \eqref{1.A}, $h(u)$ is defined
in \eqref{1.h}, $\alpha_0=\min_{i=1,\ldots,n}\pi_i^{-1} a_{i0}>0$, $\eta_0=\eta$ if
\eqref{1.wcd} holds and $\eta_0=\min_{i=1,\ldots,n}\pi_i^{-1}a_{ii}>0$
if \eqref{1.dbc} holds.
\end{lemma}

The weak formulation \eqref{1.renorm} is valid for test functions
$\phi\in C_0^\infty(\overline\Omega\times[0,T))$. We wish to allow for
test functions in $C^\infty(\overline\Omega\times[0,s])$ for some $s\in(0,T)$.

\begin{lemma}\label{lem.renorm}
Let $u$ be a renormalized solution to \eqref{1.eq}-\eqref{1.A} and let $s\in(0,T)$.
Then for any $\xi\in C^\infty([0,\infty)^n)$ with
$\xi'\in C_0^\infty([0,\infty)^n;\R^n)$ and all $\phi\in
C^\infty(\overline\Omega\times[0,s])$,
\begin{align}
  -\int_0^s\int_\Omega&\xi(u)\pa_t\phi dxdt
	+ \int_\Omega\xi(u(x,s))\phi(x,s)dx
	- \int_\Omega\xi(u^0(x))\phi(x,0)dx \nonumber \\
  &= -\sum_{i,k=1}^n\int_0^s\int_\Omega\phi\pa_i\pa_k\xi(u)\bigg(
	\sum_{j=1}^nA_{ij}(u)\na u_j - u_ib_i\bigg)\cdot\na u_kdxdt \nonumber \\
  &\phantom{xx}{}- \sum_{i=1}^n\int_0^s\int_\Omega\pa_i\xi(u)\bigg(
	\sum_{j=1}^nA_{ij}(u)\na u_j - u_ib_i\bigg)\cdot\na \phi dxdt \nonumber \\
  &\phantom{xx}{}+ \sum_{i=1}^n\int_0^s\int_\Omega\phi\pa_i\xi(u)f_i(u) dxdt.
	\label{2.renorm}
\end{align}
This expression also holds for all
$\phi\in W^{1,p}(\Omega\times(0,s))$ with $p>d+1$.
\end{lemma}

The proof of the lemma is the same as in step 1 of the proof of Lemma 11 in \cite{ChJu17}.

To define the cutoff function,
let $\varphi\in C^\infty(\R)$ be a nonincreasing
function satisfying $\varphi(x)=1$ for $x\le 0$ and $\varphi(x)=0$ for $x\ge 1$
and let $K$, $L\in\N$ with $K\ge 3$. We define
\begin{equation}\label{def.phi}
   \varphi_K^L(v) := \varphi\bigg(\frac{\log\log(\sum_{k=1}^n v_k+e)
	- \log\log(L+e)}{\log(K+1)}\bigg) \quad\mbox{for }v\in[0,\infty)^n,
\end{equation}
where $e=\exp(1)$ is the Euler number.
This function has the following properties.

\begin{lemma}\label{lem.varphi}
It holds $\varphi_K^L\in C_0^\infty([0,\infty)^n)$. Let $v\in[0,\infty)^n$. Then
\begin{enumerate}
\item[\rm (L1)] $0\le\varphi_K^L(v)\le 1$ for $v\in[0,\infty)^n$.
\item[\rm (L2)] If $\sum_{k=1}^n v_k\le L$ then $\varphi_K^L(v)=1$.
\item[\rm (L3)] If $\sum_{k=1}^n v_k> (L+e)^{K+1}$ then $\varphi_K^L(v)=0$.
\item[\rm (L4)] There exists $C>0$ such that for $v\in[0,\infty)^n$ and $j=1,\ldots,n$,
$$
  |\pa_j\varphi_K^L(v)| \le \frac{C}{\log(K+1)(\sum_{k=1}^n v_k+e)
	\log(\sum_{k=1}^n v_k+e)}.
$$
\item[\rm (L5)] There exists $C>0$ such that for $v\in[0,\infty)^n$ and $i,j=1,\ldots,n$,
$$
  |\pa_i\pa_j\varphi_K^L(v)| \le \frac{C}{\log(K+1)(\sum_{k=1}^n v_k+e)^2
	\log(\sum_{k=1}^n v_k+e)}.
$$
\end{enumerate}
\end{lemma}

\begin{proof}
If $\sum_{k=1}^nv_k\le L$ then the argument of $\varphi$ in definition \eqref{def.phi}
is negative which implies that $\varphi_K^L(v)=0$, proving (L2). Next,
$\varphi_K^L(v)=0$ holds if and only if the argument of $\varphi$ is equal or larger
than one which is equivalent to
$$
  \log\frac{\log(\sum_{k=1}^n v_k+e)}{\log(L+e)}
	= \log\log\bigg(\sum_{k=1}^n v_k+e\bigg) - \log\log(L+e) \ge \log(K+1),
$$
and, after taking the exponential, to $\sum_{k=1}^n v_k+e\ge (L+e)^{K+1}$.
This holds true since we assumed that $\sum_{k=1}^n v_k> (L+e)^{K+1}$, showing (L3).
Finally, (L4) and (L5) follow from
\begin{align*}
  \pa_j\varphi_K^L(v) &= \frac{\varphi'(z)}{\log(K+1)(\sum_{k+1}^nv_k+e)
	\log(\sum_{k=1}^n+e)}, \\
	\pa_i\pa_j\varphi_K^L(v) &= \frac{\varphi''(z)}{(\log(K+1))^2(\sum_{k+1}^nv_k+e)^2
	(\log(\sum_{k=1}^n+e))^2} \\
	&\phantom{xx}{}- \frac{\varphi'(z)}{\log(K+1)(\sum_{k+1}^nv_k+e)^2\log(\sum_{k=1}^n+e)}
	\\
	&\phantom{xx}{}- \frac{\varphi'(z)}{\log(K+1)(\sum_{k+1}^nv_k+e)^2
	(\log(\sum_{k=1}^n+e))^2},
\end{align*}
where $z$ is the argument of $\varphi$ in definition \eqref{def.phi},
since $\log(K+1)>1$.
\end{proof}


\section{Proof of Theorem \ref{thm.ws}}\label{sec.proof}

Without loss of generality, we prove Theorem \ref{thm.ws} by setting $\pi_i=1$.
This is not a restriction since these numbers only appear when applying
Lemma \ref{lem.psd} and do not change the analysis.
We split the proof into several steps.

\subsection{Approximate entropy identity for $H_{K,\eps}^{M,L}$}

We derive an integrated analog of the entropy identity \eqref{1.dHdt}
for the approximate entropy with cutoff \eqref{1.enteps}.
We choose $\phi\equiv 1$ and
$$
  \xi(u) = \varphi_K^M(u+\eps I)\sum_{i=1}^n\Big((u_i+\eps)\big(\log(u_i+\eps)
	+ \lambda_i - 1\big) + e^{-\lambda_i}\Big)
$$
in \eqref{2.renorm}, where $\eps\in(0,1/2)$ and we recall that $I=(1,\ldots,1)\in\R^n$.
Clearly, the derivative $\xi'$ is an element of $C_0^\infty([0,\infty)^n;\R^n)$,
as required. This gives the following identity for $s\in(0,T)$:
\begin{equation}\label{2.H1}
  \int_\Omega \varphi_K^M(u+\eps I)\bigg(\sum_{i=1}^n(u_i+\eps)\big(\log(u_i+\eps)
	+ \lambda_i - 1\big) + e^{-\lambda_i}\bigg)dx\bigg|_0^s
	= G_1+\cdots+G_6,
\end{equation}
where
\begin{align*}
	G_1 &= -\sum_{i=1}^n\int_0^s\int_\Omega\varphi^M_K(u+\eps I)\bigg(
	\sum_{j=1}^nA_{ij}(u)\na u_j - u_ib_i\bigg)\cdot\frac{\na u_i}{u_i+\epsilon}dxdt, \\
  G_2 &= -\sum_{i,k=1}^n\int_0^s\int_\Omega\pa_i\pa_k\varphi^M_K(u+\eps I)
	\sum_{\ell=1}^n\Big((u_\ell+\epsilon)\big(\log(u_\ell+\eps)+\lambda_\ell-1\big)
	+ e^{-\lambda_\ell}\Big) \\
  &\phantom{xxxx}{}\times\bigg(\sum_{j=1}^nA_{ij}(u)\na u_j - u_ib_i\bigg)
	\cdot\na u_k dxdt, \\
  G_3 &= -\sum_{i,k=1}^n\int_0^s\int_\Omega\pa_k\varphi^M_K(u+\eps I)
	\big(\log (u_i+\epsilon)+\lambda_i\big)\bigg(\sum_{j=1}^nA_{ij}(u)\na u_j - u_ib_i
	\bigg)\cdot\na u_kdxdt, \\
  G_4 &= -\sum_{i,k=1}^n\int_0^s\int_\Omega\pa_i\varphi^M_K(u+\eps I)
	\big(\log (u_k+\eps)+\lambda_k\big)\bigg(\sum_{j=1}^nA_{ij}(u)\na u_j - u_ib_i\bigg)
	\cdot\na u_kdxdt, \\
  G_5 &= \sum_{i=1}^n\int_0^s\int_\Omega\pa_i\varphi^M_K(u+\eps I)
	\sum_{\ell=1}^n\bigg((u_\ell+\eps)\big(\log(u_\ell+\eps)+\lambda_\ell-1\big)
	+ e^{-\lambda_\ell}\bigg)f_i(u)dxdt, \\
  G_6 &= \sum_{i=1}^n\int_0^s\int_\Omega\varphi^M_K(u+\eps I)
	\big(\log (u_i+\eps)+\lambda_i\big)f_i(u)dxdt.
\end{align*}

Next, we choose $\phi=\log v_i+\lambda_i\in W^{1,\infty}(\Omega\times(0,s))$ and
$\xi(u) = (u_i+\eps)\varphi_K^L(u+\eps I)$ in \eqref{2.renorm}. Then
\begin{align}
  &\int_\Omega (u_i+\eps)\varphi^L_K(u+\eps I)(\log v_i+\lambda_i) dx\Big|_0^s
	- \int_0^s\int_\Omega \frac{u_i+\eps}{v_i}\varphi^L_K(u+\eps I)\pa_tv_i dxdt
	\nonumber \\
  &= -\sum_{j=1}^n\int_0^s\int_\Omega \pa_j\varphi^L_K(u+\eps I)(\log v_i+\lambda_i)
	\bigg(\sum_{\ell=1}^nA_{j\ell}(u)\na u_\ell-u_jb_j\bigg)\cdot\na u_idxdt \nonumber \\
  &\phantom{xx}{} -\sum_{j=1}^n\int_0^s\int_\Omega \pa_j\varphi^L_K(u+\eps I)
	(\log v_i+\lambda_i)\bigg(\sum_{\ell=1}^nA_{i\ell}(u)\na u_\ell-u_ib_i\bigg)
	\cdot\na u_jdxdt \nonumber \\
  &\phantom{xx}{} -\sum_{j,k=1}^n\int_0^s\int_\Omega (u_i+\eps)
	\pa_j\pa_k\varphi^L_K(u+\eps I)(\log v_i+\lambda_i) \nonumber \\
	&\phantom{xxxx}{}\times\bigg(\sum_{\ell=1}^n
	A_{j\ell}(u)\na u_\ell - u_jb_j\bigg)\cdot\na u_kdxdt \nonumber \\
  &\phantom{xx}{} -\int_0^s\int_\Omega \varphi^L_K(u+\eps I)\bigg(
	\sum_{\ell=1}^nA_{i\ell}(u)\na u_\ell - u_ib_i\bigg)\cdot\frac{\na v_i}{v_i}dxdt
	\nonumber \\
  &\phantom{xx}{} -\sum_{j=1}^n\int_0^s\int_\Omega (u_i+\eps)
	\pa_j\varphi^L_K(u+\eps I)\bigg(\sum_{\ell=1}^nA_{j\ell}(u)\na u_\ell - u_jb_j\bigg)
	\cdot\frac{\na v_i}{v_i}dxdt \nonumber \\
  &\phantom{xx}{} +\int_0^s\int_\Omega \varphi^L_K(u+\eps I)(\log v_i+\lambda_i)
	f_i(u)dxdt \nonumber \\
  &\phantom{xx}{} +\sum_{j=1}^n\int_0^s\int_\Omega (u_i+\eps)
	\pa_j\varphi^L_K(u+\eps I)(\log v_i+\lambda_i)f_j(u)dxdt. \label{2.aux}
\end{align}
We wish to replace the second integral on the left-hand side. For this, we choose
the test function $\phi=(u_i+\eps)\varphi_K^L(u+\eps I)/v_i-1\in L^2(0,s;H^1(\Omega))$
in the weak formulation \eqref{1.strong} for $v$, giving
\begin{align*}
  \int_0^s & \int_\Omega \frac{u_i+\eps}{v_i}\varphi^L_K(u+\eps I)\pa_tv_i dxdt
	- \int_\Omega {v_i}dx\Big|_0^s \\
  &= -\int_0^s\int_\Omega \frac{\varphi^L_K(u+\eps I)}{v_i}
	\bigg(\sum_{\ell=1}^nA_{i\ell}(v)\na v_\ell - v_ib_i\bigg)\cdot\na u_i dxdt \\
  &\phantom{xx}{} -\sum_{j=1}^n\int_0^s\int_\Omega\frac{u_i+\eps}{v_i}
	\pa_j\varphi^L_K(u+\eps I)\bigg(\sum_{\ell=1}^nA_{i\ell}(v)\na v_\ell - v_ib_i\bigg)
	\cdot\na u_j dxdt \\
  &\phantom{xx}{} +\int_0^s\int_\Omega\frac{u_i+\eps}{v_i^2}\varphi^L_K(u+\eps I)
	\bigg(\sum_{\ell=1}^nA_{i\ell}(v)\na v_\ell - v_ib_i\bigg)\cdot\na v_i dxdt \\
  &\phantom{xx}{}+\int_0^s\int_\Omega \bigg(\frac{u_i+\eps}{v_i}\varphi^L_K(u+\eps I)
	- 1\bigg)f_i(v)dxdt.
\end{align*}
Then, replacing the second integral on the left-hand side of \eqref{2.aux} by
the previous expression, summing the resulting equation over $i=1,\ldots,n$ and
multiplying it by $-1$, we obtain
\begin{equation}\label{2.H2}
  -\int_\Omega \sum_{i=1}^n\Big(\varphi^L_K(u+\eps I)(u_i+\eps)(\log v_i+\lambda_i)
	- v_i\Big)dx\Big|_0^s =: I_1+\cdots+I_{12},
\end{equation}
where
\begin{align*}
  I_1 &= \sum_{i,j=1}^n\int_0^s\int_\Omega \pa_j\varphi^L_K(u+\eps I)(\log v_i+\lambda_i)
	\bigg(\sum_{\ell=1}^nA_{j\ell}(u)\na u_\ell - u_jb_j\bigg)\cdot\na u_idxdt, \\
  I_2 &= \sum_{i,j=1}^n\int_0^s\int_\Omega \pa_j\varphi^L_K(u+\eps I)(\log v_i+\lambda_i)
	\bigg(\sum_{\ell=1}^nA_{i\ell}(u)\na u_\ell - u_ib_i\bigg)\cdot\na u_jdxdt, \\
  I_3 &= \sum_{i,j,k=1}^n\int_0^s\int_\Omega (u_i+\eps)\pa_j\pa_k\varphi^L_K
	(u+\eps I)(\log v_i+\lambda_i)\bigg(\sum_{\ell=1}^nA_{j\ell}(u)\na u_\ell - u_jb_j\bigg)
	\cdot\na u_kdxdt, \\
  I_4 &= \sum_{i=1}^n\int_0^s\int_\Omega \varphi^L_K(u+\eps I)
	\sum_{\ell=1}^nA_{i\ell}(u)\na u_\ell\cdot\frac{\na v_i}{v_i}dxdt, \\
  I_5 &= \sum_{i,j=1}^n\int_0^s\int_\Omega (u_i+\eps)\pa_j\varphi^L_K(u+\eps I)
	\bigg(\sum_{\ell=1}^nA_{j\ell}(u)\na u_\ell - u_jb_j\bigg)
	\cdot\frac{\na v_i}{v_i}dxdt, \\
  I_6 &= \sum_{i=1}^n\int_0^s\int_\Omega \varphi^L_K(u+\eps I)\bigg(
	\sum_{\ell=1}^nA_{i\ell}(v)\na v_\ell - v_ib_i\bigg)\cdot\frac{\na u_i}{v_i}dxdt, \\
  I_7 &= \sum_{i,j=1}^n\int_0^s\int_\Omega (u_i+\eps)\pa_j\varphi^L_K(u+\eps I)\bigg(
	\sum_{\ell=1}^nA_{i\ell}(v)\na v_\ell - v_ib_i\bigg)\cdot\frac{\na u_j}{v_i}dxdt, \\
  I_8 &= -\sum_{i=1}^n\int_0^s\int_\Omega \varphi^L_K(u+\eps I)(\log v_i+\lambda_i)
	f_i(u)dxdt, \\
  I_9 &= -\sum_{i,j=1}^n\int_0^s\int_\Omega (u_i+\eps)\pa_j\varphi^L_K(u+\eps I)
	(\log v_i+\lambda_i)f_j(u)dxdt, \\
  I_{10} &= -\sum_{i=1}^n\int_0^s\int_\Omega (u_i+\eps)\varphi^L_K(u+\eps I)
	\sum_{\ell=1}^nA_{i\ell}(v)\na v_\ell\cdot\frac{\na v_i}{v_i^2}dxdt, \\
  I_{11} &= -\sum_{i=1}^n\int_0^s\int_\Omega \bigg(\frac{u_i+\eps}{v_i}
	\varphi^L_K(u+\eps I) - 1\bigg)f_i(v)dxdt, \\
  I_{12} &= \eps\sum_{i=1}^n\int_0^s\int_\Omega \varphi^L_K(u+\eps I)
	\frac{b_i\cdot\na v_i}{v_i}dxdt.
\end{align*}

Adding \eqref{2.H1} and \eqref{2.H2} gives the desired approximated
entropy identity:
\begin{equation}\label{2.Heps}
  H_{K,\eps}^{M,L}(u|v)\Big|_0^s
	+ \sum_{i=1}^n e^{-\lambda_i}\int_\Omega\varphi_K^M(u+\eps I)dx\Big|_0^s
	= G_1+\cdots+G_6+I_1+\cdots+I_{12}.
\end{equation}


\subsection{Estimate of the reaction part}

We start by estimating the terms in \eqref{2.Heps} involving the reaction
terms $f_i(u)$, namely $G_6$, $I_8$, $I_9$, and $I_{11}$
(the remaining term $G_5$ will be treated later
when we pass to the limits $\eps\to 0$ and $M\to\infty$).

We split the integral $G_6$ into two parts:
\begin{align*}
  G_6 &= \int_0^s\int_\Omega\chi_{\{\sum_{\ell=1}^n(u_\ell+\eps)>L\}}
	\varphi^M_K(u+\eps I)\sum_{i=1}^n f_i(u)\big(\log (u_i+\eps)+\lambda_i\big)dxdt \\
  &\phantom{xx}{} +\int_0^s\int_\Omega\chi_{\{\sum_{\ell=1}^n(u_\ell+\eps)\leq L\}}
	\varphi^M_K(u+\eps I)\sum_{i=1}^n f_i(u)\big(\log (u_i+\eps)+\lambda_i\big)dxdt \\
  &=: G_{61} + G_{62},
\end{align*}
where $\chi_A$ is the characteristic function on the set $A$.
Adding and subtracting the term $f_i(u+\eps I)$ and using condition (H2.ii) gives
\begin{align*}
  G_{61} &= \int_0^s\int_\Omega\chi_{\{\sum_{\ell=1}^n(u_\ell+\eps)>L\}}
	\varphi^M_K(u+\eps I)\sum_{i=1}^nf_i(u+\eps I)\big(\log (u_i+\eps)+\lambda_i\big)
	dxdt \\
  &\phantom{xx}{} +\int_0^s\int_\Omega\chi_{\{\sum_{\ell=1}^n(u_\ell+\epsilon)> L\}}
	\varphi^M_K(u+\eps I)\sum_{i=1}^n\big(f_i(u)-f_i(u+\eps I)\big)
	\big(\log (u_i+\eps)+\lambda_i\big)dxdt \\
  &\le \int_0^s\int_\Omega\chi_{\{\sum_{\ell=1}^n(u_\ell+\eps)> L\}}
	\varphi^M_K(u+\eps I)\sum_{i=1}^n\big(f_i(u)-f_i(u+\eps I)\big)
	\big(\log (u_i+\eps)+\lambda_i\big)dxdt.
\end{align*}

We claim that for any $K>0$, there exists $C(K)>0$
such that for all $0\le s\le K$, it holds that $|\log(s+\eps)|\le C(K)(1-\log\eps)$
(recall that $\eps<1/2$). Indeed, let $1/2\le s\le K$. Then
$\log \frac12\le\log(s+\eps)\le\log(K+1)$ and consequently $|\log(s+\eps)|\le C(K)$
for $C(K)=\max\{\log 2,\log(K+1)\}$. If $0\le s\le \frac12$, we find that
$|\log(s+\eps)|=-\log(s+\eps)\le-\log\eps$, which shows the claim.

We know from (L3) that $\varphi_K^M(u+\eps I)$ vanishes if $\sum_{\ell=1}^n u_\ell$
is large enough. This allows us to apply the local Lipschitz continuity of $f_i$
from (H2). Therefore, using (L1), we infer that
\begin{equation}\label{2.G61}
  G_{61} \le C(M,K,f)\eps(1-\log\eps).
\end{equation}

For $G_{62}$, we observe that $M>L$ and (L2) imply that
$\varphi_K^M(u+\eps I)=1$ in $\{\sum_{\ell=1}^n(u_\ell+\eps)\le L\}$. Hence,
$$
  G_{62} = \int_0^s\int_\Omega\chi_{\{\sum_{\ell=1}^n(u_\ell+\eps)\leq L\}}
	\sum_{i=1}^n f_i(u)\big(\log (u_i+\eps)+\lambda_i\big)dxdt.
$$
We wish to estimate this term together with the terms $I_8$, $I_9$, and $I_{11}$.
Consider the integrands of $G_{62}$, $I_8$, and $I_{11}$ in the set
$\{\sum_{\ell=1}^n (u_\ell+\eps)\le L\}$ (where it holds that $\varphi_K^L(u+\eps I)=1$):
\begin{align*}
  f_i(u)&\big(\log (u_i+\eps)+\lambda_i\big) - f_i(u)\varphi^L_K(u+\eps I)
	(\log v_i+\lambda_i)
  - f_i(v)\bigg(\frac{u_i+\eps}{v_i}\varphi^L_K(u+\eps I) - 1\bigg) \\
  &= f_i(u)\log \frac{u_i+\epsilon}{v_i} - f_i(v)\bigg(\frac{u_i+\eps}{v_i} - 1\bigg) \\
  &= f_i(u)\bigg(\log\frac{u_i+\eps}{v_i } -\frac{u_i+\eps}{v_i} + 1\bigg)
	+ \big(f_i(u)-f_i(v)\big)\bigg(\frac{u_i+\epsilon}{v_i} - 1\bigg).
\end{align*}
Therefore, we need to estimate
\begin{align*}
  G_{62} &+ I_8 + I_9 + I_{11} \\
  &\le \sum_{i=1}^n\int_0^s\int_\Omega\chi_{\{\sum_{\ell=1}^n(u_\ell+\eps)\leq L\}}
	f_i(u)\bigg(\log\frac{u_i+\eps}{v_i}-\frac{u_i+\eps}{v_i} + 1\bigg)dxdt \\
  &\phantom{xx}{} +\sum_{i=1}^n\int_0^s\int_\Omega
	\chi_{\{\sum_{\ell=1}^n(u_\ell+\eps)\leq L\}}
	\big(f_i(u)-f_i(v)\big)\bigg(\frac{u_i+\eps}{v_i} - 1\bigg)dxdt \\
  &\phantom{xx}{}+ \sum_{i=1}^n\int_0^s\int_\Omega
	\chi_{\{\sum_{\ell=1}^n(u_\ell+\eps)> L\}} \\
	&\phantom{xxxx}{}\times\bigg(|f_i(u)\varphi^L_K(u+\eps I)|
	|\log v_i+\lambda_i| + |f_i(v)|\bigg(\frac{u_i+\eps}{v_i} + 1\bigg)\bigg)dxdt \\
  &\phantom{xx}{} +\sum_{i,j=1}^n\int_0^s\int_\Omega\big|(u_i+\eps)
	\pa_j\varphi^L_K(u+\eps I)\big||f_j(u)||\log v_i+\lambda_i|dxdt \\
  &=: J_1+\cdots+J_4.
\end{align*}

We first consider $J_1$. The elementary inequalities
$-|s-1|^2\le \log s-s+1\le 0$ for $s\ge 1$ and $-|s-1|^2/s\le\log s-s+1\le 0$
for $s\in(0,1)$ imply that (as shown in \cite{Fis17})
\begin{equation}\label{2.ineq}
  -\bigg(1+\frac{1}{s}\bigg)|s-1|^2 \le \log s-s+1 \le 0\quad\mbox{for }s>0.
\end{equation}
Furthermore, we use the local Lipschitz continuity of $f_i$ and the quasi-positivity
property $f_i(u)\ge 0$ for all $u\in[0,\infty)^n$ with $u_i=0$ (as a consequence
of (H2.ii)) to conclude that
in the set $\{\sum_{\ell=1}^n u_\ell\le L\}$,
\begin{align*}
  -f_i(u) &\le f_i(u_1,\ldots,u_{i-1},0,u_{i+1},\ldots,u_n) - f_i(u) \\
  &\le |f_i(u_1,\ldots,u_{i-1},0,u_{i+1},\ldots,u_n) - f_i(u)|
	\le C(L,f_i)u_i.
\end{align*}
This allows us to estimate the integrand of $J_1$. Indeed, we obtain in
$\{\sum_{\ell=1}^n u_\ell\le L\}$
\begin{align*}
  f_i(u)&\bigg(\log\frac{u_i+\eps}{v_i} - \frac{u_i+\eps}{v_i} + 1\bigg)
  \le C(L,f_i)u_i\bigg(1+\frac{v_i}{u_i+\eps}\bigg)
	\bigg|\frac{u_i+\eps}{v_i}-1\bigg|^2 \\
  &= C(L,f_i)\bigg(u_i+\frac{u_i}{u_i+\eps}v_i\bigg)
	\frac{1}{v_i^2}\big|(u_i-v_i)+\eps\big|^2
  \le C(L,f_i,v_i)\big(|u_i-v_i|^2+\eps^2\big).
\end{align*}
This estimate also holds in $\{\sum_{\ell=1}^n (u_\ell+\eps)\le L\}$ for $\eps>0$
since $\{\sum_{\ell=1}^n( u_\ell+\eps)\le L\}$ is a subset of
$\{\sum_{\ell=1}^n u_\ell\le L\}$. We deduce that
\begin{align*}
  J_1
  &\le \int_0^s\int_\Omega\chi_{\{\sum_{\ell=1}^nu_\ell\leq L\}}
	\sum_{i=1}^n C(L,f_i,v_i)\big(|u_i-v_i|^2+\eps^2\big)dxdt \\
  &\le C(L,f,v)\int_0^s\int_\Omega\chi_{\{\sum_{\ell=1}^nu_\ell\leq L\}}
	\sum_{i=1}^n|u_i-v_i|^2dxdt + C(L,f,v,T,\Omega)\eps^2.
\end{align*}
Using again $\{\sum_{\ell=1}^n (u_\ell+\eps)\le L\}\subset\{\sum_{\ell=1}^n u_\ell\le L\}$
and the local Lipschitz continuity of $f_i$, it follows that
\begin{align*}
  J_2 &\le \sum_{i=1}^n\int_0^s\int_\Omega\chi_{\{\sum_{\ell=1}^nu_\ell\leq L\}}
	C(L,f_i,v_i)|u-v|\big(|u_i-v_i|+\eps\big)dxdt \\
  &\le C(L,f,v)\int_0^s\int_\Omega\chi_{\{\sum_{\ell=1}^nu_\ell\leq L\}}
	\sum_{i=1}^n|u_i-v_i|^2dxdt + C(L,f,v,T,\Omega)\eps.
\end{align*}
Taking into account (L3), we have $|f_i(u)\varphi^L_K(u+\epsilon I)|\leq C(L,K,f_i)$
and thus
$$
  J_3 \le C(L,K,f,v)\sum_{i=1}^n\int_0^s\int_\Omega
	\chi_{\{\sum_{\ell=1}^n(u_\ell+\eps)> L\}}\bigg(1+\sum_{i=1}^nu_i\bigg)dxdt.
$$
Since $\pa_j\varphi_K^L(u+\eps I)=0$ for sufficiently large $u$, we can estimate as
$$
  J_4 \le C(L,K,f,v)\int_0^s\int_\Omega\chi_{\{\sum_{\ell=1}^n(u_\ell+\eps)> L\}}dxdt.
$$
We conclude that
\begin{align}
  G_{6} &+ I_8 + I_9 + I_{11}
	\le C(L,f,v)\int_0^s\int_\Omega\chi_{\{\sum_{\ell=1}^nu_\ell\leq L\}}
	\sum_{i=1}^n|u_i-v_i|^2dxdt \nonumber \\
	&{}+ C(L,K,f,v)\sum_{i=1}^n\int_0^s\int_\Omega
	\chi_{\{\sum_{\ell=1}^n(u_\ell+\eps)> L\}}\bigg(1+\sum_{i=1}^nu_i\bigg)dxdt
	\nonumber \\
	&{}+ C(M,K,f)\eps(1-\log\eps)	+ C(L,f,v,T,\Omega)\eps. \label{2.conc1}
\end{align}


\subsection{Estimate of the cross-diffusion part}

We estimate only some terms involving the diffusion coefficients, namely
$G_1$, $I_4$, $I_6$, and $I_{10}$.
We split $G_1=G_{11}+G_{12}$ in \eqref{2.H1} and $I_6=I_{61}+I_{61}$ in \eqref{2.H2}
into two parts, the cross-diffusion part and the drift part:
\begin{align*}
  G_{11} &= -\sum_{i,j=1}^n\int_0^s\int_\Omega\varphi^M_K(u+\eps I)A_{ij}(u)
	\na u_j\cdot\frac{\na u_i}{u_i+\eps}dxdt, \\
  G_{12} &= \sum_{i=1}^n\int_0^s\int_\Omega\varphi^M_K(u+\epsilon I)
	u_ib_i\cdot\frac{\na u_i}{u_i+\epsilon}dxdt, \\
	I_{61} &= \sum_{i,j=1}^n\int_0^s\int_\Omega\varphi_K^L(u+\eps I)A_{ij}(v)\na v_j
	\cdot\frac{\na u_i}{v_i}dxdt, \\
	I_{62} &= -\sum_{i=1}^n\int_0^s\int_\Omega\varphi_K^L(u+\eps I)v_ib_i
	\cdot\frac{\na u_i}{v_i}dxdt.
\end{align*}
We split $\Omega$ into the subsets $\{\sum_{\ell=1}^n(u_\ell+\eps)\leq L\}$ and
$\{\sum_{\ell=1}^n(u_\ell+\eps)>L\}$ and combine on the former set
the terms $G_{11}+I_4$ and $I_{61}+I_{10}$. This yields
\begin{align}
  G_{11} &+ I_4 + I_{61} + I_{10}
	= -\sum_{i,j=1}^n\int_0^s\int_\Omega \chi_{\{\sum_{\ell=1}^n(u_\ell+\eps)\leq L\}}
	\bigg\{A_{ij}(u)\na u_j\cdot\bigg(\frac{\na u_i}{u_i+\eps} - \frac{\na v_i}{v_i}\bigg)
	\nonumber \\
  &\phantom{xxxx}{} +A_{ij}(v)\na v_j\cdot\bigg(\frac{\na v_i}{v_i}\frac{u_i+\eps}{v_i}
	- \frac{\na u_i}{v_i}\bigg)\bigg\}dxdt \nonumber \\
  &\phantom{xx}{} -\sum_{i,j=1}^n\int_0^s\int_\Omega
	\chi_{\{\sum_{\ell=1}^n(u_\ell+\eps)> L\}}\varphi^M_K(u+\eps I)A_{ij}(u)
	\na u_j\cdot\frac{\na u_i}{u_i+\eps}dxdt \nonumber \\
  &\phantom{xx}{} +\sum_{i,j=1}^n\int_0^s\int_\Omega
	\chi_{\{\sum_{\ell=1}^n(u_\ell+\eps)> L\}}\varphi^L_K(u+\eps I)A_{ij}(u)
	\na u_j\cdot\frac{\na v_i}{v_i}dxdt \nonumber \\
  &\phantom{xx}{} +\sum_{i,j=1}^n\int_0^s\int_\Omega
	\chi_{\{\sum_{\ell=1}^n(u_\ell+\eps)> L\}}\varphi^L_K(u+\eps I)A_{ij}(v)
	\na v_j\cdot\frac{\na u_i}{v_i}dxdt \nonumber \\
  &\phantom{xx}{} -\sum_{i,j=1}^n\int_0^s\int_\Omega
	\chi_{\{\sum_{\ell=1}^n(u_\ell+\eps)> L\}}\varphi^L_K(u+\eps I)A_{ij}(v)
	\na v_j\cdot\frac{\na v_i}{v_i}\frac{u_i+\eps}{v_i}dxdt \nonumber \\
  &=: O_1 + \cdots + O_5. \label{2.defO}
\end{align}
The estimation of the expressions $O_i$ is rather technical. We start with $O_1$.

\medskip\noindent{\bf Estimation of $O_1$.}
We add and subtract $A_{ij}(u+\eps I)$ in $O_1$, which gives $O_1=O_{11}+O_{12}$,
where
\begin{align*}
  O_{11} &= -\sum_{i,j=1}^n\int_0^s\int_\Omega
	\chi_{\{\sum_{\ell=1}^n(u_\ell+\eps)\leq L\}}
	\bigg\{A_{ij}(u+\eps I)\na u_j\cdot\bigg(\frac{\na u_i}{u_i+\eps}
	- \frac{\na v_i}{v_i}\bigg) \\
  &\phantom{xxxx}{} +A_{ij}(v)\na v_j\cdot\bigg(\frac{\na v_i}{v_i}\frac{u_i+\eps}{v_i}
	- \frac{\na u_i}{v_i}\bigg)\bigg\}dxdt, \\
  O_{12} &= \sum_{i,j=1}^n\int_0^s\int_\Omega
	\chi_{\{\sum_{\ell=1}^n(u_\ell+\eps)\leq L\}} \big(A_{ij}(u+\eps I)-A_{ij}(u)\big)
	\na u_j\cdot\bigg(\frac{\na u_i}{u_i+\eps}-\frac{\na v_i}{v_i}\bigg)dxdt.
\end{align*}
Furthermore, we add and subtract the term $\na v_j/v_j$ in $O_{11}$.
We find after a short computation that
\begin{align*}
  O_{11} &= -\sum_{i,j=1}^n\int_0^s\int_\Omega
	\chi_{\{\sum_{\ell=1}^n(u_\ell+\eps)\leq L\}} A_{ij}(u+\eps I)(u_j+\eps) \\
	&\phantom{xxxx}{} \times\bigg(\frac{\na u_i}{u_i+\eps}-\frac{\na v_i}{v_i}\bigg)
	\cdot\bigg(\frac{\na u_j}{u_j+\eps}-\frac{\na v_j}{v_j}\bigg)dxdt \\
  &\phantom{xx}{} -\sum_{i,j=1}^n\int_0^s\int_\Omega
	\chi_{\{\sum_{\ell=1}^n(u_\ell+\eps)\leq L\}} A_{ij}(u+\eps I)(u_j+\eps)
	\frac{\na v_j}{v_j}\cdot\bigg(\frac{\na u_i}{u_i+\eps}-\frac{\na v_i}{v_i}\bigg)
	dxdt \\
  &\phantom{xx}{} -\sum_{i,j=1}^n\int_0^s\int_\Omega
	\chi_{\{\sum_{\ell=1}^n(u_\ell+\eps)\leq L\}} A_{ij}(v)\na v_j
	\cdot\bigg(\frac{\na v_i}{v_i}\frac{u_i+\eps}{v_i}-\frac{\na u_i}{v_i}\bigg)dxdt \\
  &= -\sum_{i,j=1}^n\int_0^s\int_\Omega
	\chi_{\{\sum_{\ell=1}^n(u_\ell+\eps)\leq L\}} A_{ij}(u+\eps I)(u_j+\eps) \\
	&\phantom{xxxx}{} \times\bigg(\frac{\na u_i}{u_i+\eps}-\frac{\na v_i}{v_i}\bigg)
	\cdot\bigg(\frac{\na u_j}{u_j+\eps}-\frac{\na v_j}{v_j}\bigg)dxdt \\
  &\phantom{xx}{} -\sum_{i,j=1}^n\int_0^s\int_\Omega
	\chi_{\{\sum_{\ell=1}^n(u_\ell+\eps)\leq L\}}
	\bigg(A_{ij}(u+\eps I)\frac{u_j+\eps}{v_j}-A_{ij}(v)\frac{u_i+\eps}{v_i}\bigg) \\
	&\phantom{xxxx}{}\times \na v_j\cdot\bigg(
	\frac{\na u_i}{u_i+\eps}-\frac{\na v_i}{v_i}\bigg)dxdt \\
  &=: O_{111}+O_{112}.
\end{align*}
It follows from the positive definiteness of $A(u)h''(u)^{-1}$ (Lemma \ref{lem.psd}) that
\begin{align}
  O_{111} &\le -\alpha_0\sum_{i=1}^n \int_0^s\int_\Omega
	\chi_{\{\sum_{\ell=1}^n(u_\ell+\eps)\leq L\}} (u_i+\eps)
	\bigg|\frac{\na u_i}{u_i+\eps}-\frac{\na v_i}{v_i}\bigg|^2dxdt \nonumber \\
  &\phantom{xx}{} -2\eta_0\sum_{i=1}^n \int_0^s\int_\Omega
	\chi_{\{\sum_{\ell=1}^n(u_\ell+\eps)\leq L\}} (u_i+\eps)^2
	\bigg|\frac{\na u_i}{u_i+\eps}-\frac{\na v_i}{v_i}\bigg|^2dxdt. \label{2.O111}
\end{align}
For the estimate of $O_{112}$, we use definition \eqref{1.A}
of the coefficients $A_{ij}$.
Some terms cancel in $O_{112}$ and we end up with
\begin{align*}
  O_{112} &= -\sum_{i=1}^n\int_0^s\int_\Omega
	\chi_{\{\sum_{\ell=1}^n(u_\ell+\eps)\leq L\}}
  \sum_{k=1}^n a_{ik}(u_k-v_k+\eps)\frac{u_i+\eps}{v_i}\na v_i\cdot
	\bigg(\frac{\na u_i}{u_i+\eps} - \frac{\na v_i}{v_i}\bigg)dxdt \\
  &\phantom{xx}{} -\sum_{i,j=1}^n\int_0^s\int_\Omega
	\chi_{\{\sum_{\ell=1}^n(u_\ell+\eps)\leq L\}} a_{ij}
	\bigg((u_i+\eps) \frac{u_j+\eps}{v_j}-v_i\frac{u_i+\eps}{v_i}\bigg)\\
	&\phantom{xxxx}{}\times \na v_j
	\cdot\bigg(\frac{\na u_i}{u_i+\eps}-\frac{\na v_i}{v_i}\bigg)dxdt \\
  &= -\sum_{i,j=1}^n\int_0^s\int_\Omega \chi_{\{\sum_{\ell=1}^n(u_\ell+\eps)\leq L\}}
	a_{ij}(u_j-v_j+\eps)(u_i+\eps) \\
	&\phantom{xxxx}{}\times\bigg(\frac{\na v_i}{v_i}+\frac{\na v_j}{v_j}\bigg)
	\cdot\bigg(\frac{\na u_i}{u_i+\eps}-\frac{\na v_i}{v_i}\bigg)dxdt.
\end{align*}
Using the regularity of $v$ and Young's inequality, we find that
\begin{align*}
  O_{112} &\le C(v)\sum_{i,j=1}^n\int_0^s\int_\Omega
	\chi_{\{\sum_{\ell=1}^n(u_\ell+\eps)\leq L\}} |u_j-v_j+\eps|(u_i+\eps)
	\bigg|\frac{\na u_i}{u_i+\eps}-\frac{\na v_i}{v_i}\bigg|dxdt \\
  &\le \eta_0\sum_{i=1}^n \int_0^s\int_\Omega
	\chi_{\{\sum_{\ell=1}^n(u_\ell+\eps)\leq L\}} (u_i+\eps)^2
	\Bigg|\frac{\na u_i}{u_i+\eps}-\frac{\na v_i}{v_i}\bigg|^2dxdt \\
  &\phantom{xx}{} +C(v)\sum_{i=1}^n\int_0^s\int_\Omega
	\chi_{\{\sum_{\ell=1}^n u_\ell\leq L\}} |u_i-v_i|^2 dxdt + C(v,T,\Omega)\eps^2,
\end{align*}
where in the last step we have used $\{\sum_{\ell=1}^n(u_\ell+\eps)\leq L\}
\subset\{\sum_{\ell=1}^n u_\ell\leq L\}$. The first term on the right-hand side
can be absorbed by the second term on the right-hand side of estimate \eqref{2.O111}
for $O_{111}$, and combining the estimates, we obtain
$$
  O_{11} \le C(v)\sum_{i=1}^n\int_0^s\int_\Omega \chi_{\{\sum_{\ell=1}^nu_\ell\leq L\}}
	|u_i-v_i|^2dxdt + C(v,T,\Omega)\eps^2.
$$

We turn to the estimate of $O_{12}$. Again using definition \eqref{1.A} of $A_{ij}$,
it follows that
\begin{align*}
  O_{12} &= \eps\sum_{i,j=1}^n\int_0^s\int_\Omega
	\chi_{\{\sum_{\ell=1}^n(u_\ell+\eps)\leq L\}}
	\bigg(\delta_{ij}\sum_{k=1}^na_{ik}+a_{ij}\bigg)\na u_j
	\cdot\bigg(\frac{\na u_i}{u_i+\eps}-\frac{\na v_i}{v_i}\bigg)dxdt \\
  &\le C\eps\sum_{i,j=1}^n\int_0^s\int_\Omega
	\chi_{\{\sum_{\ell=1}^n(u_\ell+\eps)\leq L\}} |\na u_j|\frac{|\na u_i|}{u_i+\eps}dxdt \\
  &\phantom{xx}{}+ C(v)\eps\sum_{i=1}^n\int_0^s\int_\Omega
	\chi_{\{\sum_{\ell=1}^n(u_\ell+\eps)\leq L\}} |\na u_i|dxdt.
\end{align*}
We integrand of the first term on the right-hand side can be reformulated
according to
\begin{equation}\label{2.sqrt}
  \sqrt{\eps}|\na u_j|\frac{|\na u_i|}{u_i+\eps}
	= 2\sqrt{\frac{\eps}{u_i+\eps}}\sqrt{\frac{u_i}{u_i+\eps}}|\na {u_j}||\na \sqrt{u_i}|
	\leq 2|\na {u_j}||\na \sqrt{u_i}|,
\end{equation}
and using Young's inequality, we deduce that
\begin{align*}
  O_{12} &\le C(v)\sqrt{\eps}\sum_{i=1}^n\int_0^s\int_\Omega |\na u_i|^2dxdt
  + C(v)\sqrt{\eps}\sum_{i=1}^n\int_0^s\int_\Omega\big(1+|\na \sqrt{u_i}|^2\big)dxdt \\
  &\le C(v,T,\Omega)\sqrt{\eps}.
\end{align*}
Note that we need here the condition $a_{i0}>0$ which yields an $L^2$ bound for
$\na\sqrt{u_i}$ (see Remark \ref{rem}).

We conclude the estimate of $O_1$ by adding the bounds for $O_{11}$ and $O_{12}$:
$$
  O_1 \le C(v)\sum_{i=1}^n\int_0^s\int_\Omega
	\chi_{\{\sum_{\ell=1}^nu_\ell\leq L\}} |u_i-v_i|^2dxdt + C(v,T,\Omega)\sqrt{\eps}.
$$

\medskip\noindent{\bf Estimation of $O_2$.} We add and subtract $A_{ij}(u+\eps I)$
in definition \eqref{2.defO} of $O_2$ and use the definition of $A_{ij}$ to find that
\begin{align*}
  O_2 &= -\sum_{i,j=1}^n\int_0^s\int_\Omega\chi_{\{\sum_{\ell=1}^n(u_\ell+\eps)> L\}}
	\varphi^M_K(u+\eps I) A_{ij}(u+\eps I)\na u_j\cdot\frac{\na u_i}{u_i+\eps}dxdt \\
  &\phantom{xx}{}+\sum_{i,j=1}^n\int_0^s\int_\Omega
	\chi_{\{\sum_{\ell=1}^n(u_\ell+\eps)> L\}}\varphi^M_K(u+\eps I)
	\big(A_{ij}(u+\eps I)-A_{ij}(u)\big)\na u_j\cdot\frac{\na u_i}{u_i+\eps}dxdt \\
  &= -\sum_{i,j=1}^n\int_0^s\int_\Omega\chi_{\{\sum_{\ell=1}^n(u_\ell+\eps)> L\}}
	\varphi^M_K(u+\eps I) A_{ij}(u+\eps I)(u_j+\eps)\frac{\na u_j}{u_j+\eps}
	\cdot\frac{\na u_i}{u_i+\eps}dxdt \\
  &\phantom{xx}{} +\eps\sum_{i,j=1}^n\int_0^s\int_\Omega
	\chi_{\{\sum_{\ell=1}^n(u_\ell+\eps)> L\}}\varphi^M_K(u+\eps I)
	\bigg(\delta_{ij}\sum_{k=1}^na_{ik}+a_{ij}\bigg)
	\na u_j\cdot\frac{\na u_i}{u_i+\eps}dxdt \\
  &=: O_{21}+O_{22}.
\end{align*}
We employ the positive definiteness of $A(u+\eps I)h''(u+\eps I)^{-1}$
to estimate $O_{21}$:
\begin{align*}
  O_{21} &\le -2\eta_0\sum_{i=1}^n\int_0^s\int_\Omega
	\chi_{\{\sum_{\ell=1}^n(u_\ell+\eps)> L\}}\varphi^M_K(u+\eps I)
	(u_i+\eps)^2\bigg|\frac{\na u_i}{u_i+\eps}\bigg|^2dxdt \\
	&\le -2\eta_0\sum_{i=1}^n\int_0^s\int_\Omega
	\chi_{\{\sum_{\ell=1}^n(u_\ell+\eps)> L\}}\varphi^M_K(u+\eps I)|\na u_i|^2dxdt.
\end{align*}
For the estimate of $O_{22}$, we take into account \eqref{2.sqrt} and use
Young's inequality similarly as in the estimate of $O_{12}$:
\begin{align*}
  O_{22} &\le C\eps\sum_{i,j=1}^n\int_0^s\int_\Omega|\na u_j|
	\bigg|\frac{\na u_i}{u_i+\eps}\bigg|dxdt \\
  &\le C\sqrt{\eps}\sum_{i=1}^n\int_0^s\int_\Omega\big(|\na u_i|^2
	+ |\na \sqrt{u_i}|^2\big)dxdt
	\le C\sqrt{\eps}.
\end{align*}
Adding the inequalities for $O_{21}$ and $O_{22}$ then gives
$$
  O_2 \le -2\eta_0\sum_{i=1}^n\int_0^s\int_\Omega
	\chi_{\{\sum_{\ell=1}^n(u_\ell+\eps)> L\}}\varphi^M_K(u+\eps I)|\na u_i|^2dxdt
	+ C\sqrt{\eps}.
$$

\medskip\noindent{\bf Estimation of $O_3$, $O_4$, and $O_5$.}
We conclude from (L3) and Young's inequality that
\begin{align*}
  O_3 &= \sum_{i,j=1}^n\int_0^s\int_\Omega
	\chi_{\{(L+e)^{K+1}>\sum_{\ell=1}^n(u_\ell+\eps)> L\}}
	\varphi^L_K(u+\eps I)A_{ij}(u)\na u_j\cdot\frac{\na v_i}{v_i}dxdt \\
  &\le C(L,K,v)\sum_{i=1}^n\int_0^s\int_\Omega
	\chi_{\{(L+e)^{K+1}>\sum_{\ell=1}^n(u_\ell+\eps)> L\}}|\na {u_i}|dxdt \\
  &\le \frac{\eta_0}{2}\sum_{i=1}^n\int_0^s\int_\Omega
	\chi_{\{\sum_{\ell=1}^n(u_\ell+\eps)> L\}}|\na {u_i}|^2dxdt \\
	&\phantom{xx}{}
	+ C(L,K,v)\sum_{i=1}^n\int_0^s\int_\Omega\chi_{\{\sum_{\ell=1}^n(u_\ell+\eps)> L\}}dxdt.
\end{align*}
In a similar way, we can estimate
\begin{align*}
  O_4+O_5 &= \sum_{i,j=1}^n\int_0^s\int_\Omega\chi_{\{\sum_{\ell=1}^n(u_\ell+\eps)> L\}}
	\varphi^L_K(u+\eps I)A_{ij}(v)\na v_j\cdot\frac{\na u_i}{v_i}dxdt \\
  &\phantom{xx}{} -\sum_{i,j=1}^n\int_0^s\int_\Omega
	\chi_{\{\sum_{\ell=1}^n(u_\ell+\eps)> L\}}\varphi^L_K(u+\eps I)A_{ij}(v)
	\na v_j\cdot\frac{\na v_i}{v_i}\frac{u_i+\eps}{v_i}dxdt \\
  &\le C(v)\sum_{i=1}^n\int_0^s\int_\Omega\chi_{\{\sum_{\ell=1}^n(u_\ell+\eps)> L\}}
	|\na{u_i}|dxdt \\
  &\phantom{xx}{} +C(v)\sum_{i=1}^n\int_0^s\int_\Omega
	\chi_{\{\sum_{\ell=1}^n(u_\ell+\eps)> L\}}(u_i+1)dxdt \\
  &\le \frac{\eta_0}{2}\sum_{i=1}^n\int_0^s\int_\Omega
	\chi_{\{\sum_{\ell=1}^n(u_\ell+\eps)> L\}}|\na {u_i}|^2dxdt \\
  &\phantom{xx}{} +C(v)\int_0^s\int_\Omega\chi_{\{\sum_{\ell=1}^n(u_\ell+\eps)> L\}}
	\bigg(1+\sum_{i=1}^nu_i\bigg)dxdt.
\end{align*}

Adding all the estimates for $O_1,\ldots,O_5$, we conclude from
\eqref{2.defO} that
\begin{align}
  G_{11}&+I_4+I_{61}+I_{10} \nonumber \\
  &\le -2\eta_0\sum_{i=1}^n\int_0^s\int_\Omega\chi_{\{\sum_{\ell=1}^n(u_\ell+\eps)> L\}}
	\varphi^M_K(u+\eps I)|\na u_i|^2dxdt \nonumber \\
  &\phantom{xx}{} +\eta_0\sum_{i=1}^n\int_0^s\int_\Omega
	\chi_{\{\sum_{\ell=1}^n(u_\ell+\eps)> L\}}|\na {u_i}|^2dxdt \nonumber \\
  &\phantom{xx}{} +C(v)\sum_{i=1}^n\int_0^s\int_\Omega
	\chi_{\{\sum_{\ell=1}^nu_\ell\leq L\}} |u_i-v_i|^2dxdt \nonumber \\
  &\phantom{xx}{} +C(L,K,v)\int_0^s\int_\Omega\chi_{\{\sum_{\ell=1}^n(u_\ell+\eps)> L\}}
	\bigg(1+\sum_{i=1}^nu_i\bigg)dxdt \nonumber \\
  &\phantom{xx}{} +C(v,T,\Omega)\sqrt{\eps}. \label{2.conc2}
\end{align}


\subsection{The limit $\eps\to 0$.}

Inserting the estimates of the previous subsections and observing that the
term $I_{12}$ can be estimated as
$$
  I_{12} = \eps\sum_{i=1}^n\int_0^s \int_\Omega\varphi_K^L(u+\eps I)
	\frac{b_i\cdot\na v_i}{v_i}dxdt \le C(b,v,T,\Omega)\eps,
$$
we infer from \eqref{2.Heps} that
\begin{align}
  H_{K,\epsilon}^{M, L}&(u|v)\Big|_0^s
	+ \sum_{i=1}^n e^{-\lambda_i}\int_\Omega\varphi_K^M(u+\eps I)dx\Big|_0^s \nonumber \\
	&\le -2\eta_0\sum_{i=1}^n\int_0^s\int_\Omega\chi_{\{\sum_{\ell=1}^n(u_\ell+\eps)> L\}}
	\varphi^M_K(u+\eps I)|\na u_i|^2dxdt \nonumber \\
  &\phantom{xx}{} +\eta_0\sum_{i=1}^n\int_0^s\int_\Omega
	\chi_{\{\sum_{\ell=1}^n(u_\ell+\eps)> L\}}|\na {u_i}|^2dxdt \nonumber \\
	&\phantom{xx}{} +C(L,K,f,v)\int_0^s\int_\Omega
	\chi_{\{\sum_{\ell=1}^n(u_\ell+\eps)> L\}}\bigg(1+\sum_{i=1}^nu_i\bigg)dxdt \nonumber \\
  &\phantom{xx}{} +C(L,f,v)\sum_{i=1}^n\int_0^s\int_\Omega
	\chi_{\{\sum_{\ell=1}^nu_\ell\leq L\}} |u_i-v_i|^2dxdt \nonumber \\
  &\phantom{xx}{} +C(M,K)\eps(1-\log\eps) + C(L,b,f,v,T,\Omega)\sqrt{\eps} \nonumber \\
  &\phantom{xx}{} +G_{12}+G_2+\cdots+G_5+I_1+I_2+I_{3}+I_5+I_{62}+I_7.
	\label{2.conc3}
\end{align}
We pass to the limit $\eps\to 0$ in this inequality.
First, we consider the left-hand side. We split the integral
of $H_{K,\eps}^{M,l}$ into two parts and analyze each part separately.
By the mean-value theorem, we have for some $\theta_i\in[0,1]$,
\begin{align*}
  \sum_{i=1}^n & (u_i+\eps)\big(\log (u_i+\eps)+\lambda_i-1\big)
	+ \sum_{i=1}^ne^{-\lambda_i} = h(u+\eps I) \\
  &= h(u)+\sum_{i=1}^n h_i'(u_i+\theta_i\eps)\eps
	= h(u)+\eps\sum_{i=1}^n \big(\log(u_i+\theta_i\eps)+\lambda_i\big) \\
  &\le h(u)+\sum_{i=1}^n (u_i+1+\lambda_i)\in L^1(\Omega).
\end{align*}
Thus, together with the bound (L1), we can apply the dominated convergence theorem
to infer that, as $\eps\to 0$,
\begin{align*}
  \int_\Omega &\varphi^M_K(u+\eps I)\sum_{i=1}^n\big[(u_i+\eps)\big(
	\log (u_i+\eps)+\lambda_i-1\big) + e^{-\lambda_i}\big]dx\Big|_0^s \\
  &\to \int_\Omega \varphi^M_K(u)\sum_{i=1}^n \big[u_i \big(\log  u_i +\lambda_i-1\big)
	+ e^{-\lambda_i}\big]dx\Big|_0^s.
\end{align*}
Similarly, it follows from the uniform bound
$$
  \bigg|\sum_{i=1}^n\Big(\varphi^L_K(u+\eps I)(u_i+\eps)(\log v_i+\lambda_i)
	-v_i)\bigg|
	\leq C(v)\bigg(\sum_{i=1}^nu_i+1\bigg)\in L^1(\Omega)
$$
that in the limit $\eps\to 0$,
\begin{align*}
  \int_\Omega \sum_{i=1}^n&\Big(\varphi^L_K(u+\eps I)(u_i+\eps)(\log v_i+\lambda_i)
	- v_i\Big)dx\Big|_0^s \\
	&\to \int_\Omega \sum_{i=1}^n\big(\varphi^L_K(u)u_i(\log v_i+\lambda_i)
	- v_i\big)dx\Big|_0^s.
\end{align*}
Consequently, the left-hand side of \eqref{2.conc3} converges as $\eps\to 0$:
\begin{align}
  H_{K,\eps}^{M,L}&(u|v)
	+ \sum_{i=1}^n e^{-\lambda_i}\int_\Omega\varphi_K^M(u+\eps I)dx\Big|_0^s \nonumber \\
	&\to H_K^{M,L}(u|v) + \sum_{i=1}^n e^{-\lambda_i}\int_\Omega\varphi_K^M(u)dx
	\Big|_0^s, \label{2.HKLM}
\end{align}
where
$$
  H_K^{M,L}(u|v)
	:= \int_\Omega\sum_{i=1}^n\Big(\varphi_K^M(u)u_i(\log u_i+\lambda_i-1)
	- \varphi_K^L(u)u_i(\log v_i+\lambda_i) + v_i\Big)dx.
$$

Next, we turn to the limit $\eps\to 0$ on the right-hand side of
\eqref{2.conc3}. We observe that for a.e.\ $(x,t)\in\Omega\times(0,s)$,
$$
  \lim_{\eps\to 0}\chi_{\{\sum_{\ell=1}^n(u_\ell+\eps)> L\}}(x,t)
	= \chi_{\{\sum_{\ell=1}^nu_\ell\geq L\}}(x,t).
$$
Then, by the dominated convergence theorem,
we can pass to the limit $\eps\to 0$ in the first three terms on the right-hand side
of \eqref{2.conc3}, leading to
\begin{align*}
  & \sum_{i=1}^n\int_0^s\int_\Omega\chi_{\{\sum_{\ell=1}^n(u_\ell+\eps)> L\}}
	\varphi^M_K(u+\eps I)|\na {u_i}|^2dxdt \\
  &\phantom{xxxxxx}{}\to
	\sum_{i=1}^n\int_0^s\int_\Omega\chi_{\{\sum_{\ell=1}^n u_\ell\geq L\}}
	\varphi^M_K(u)|\na {u_i}|^2dxdt, \\
	& \sum_{i=1}^n\int_0^s\int_\Omega\chi_{\{\sum_{\ell=1}^n(u_\ell+\eps)> L\}}
	|\na {u_i}|^2dxdt
  \to \sum_{i=1}^n\int_0^s\int_\Omega\chi_{\{\sum_{\ell=1}^n u_\ell\geq L\}}
	|\na {u_i}|^2dxdt, \\
  & \int_0^s\int_\Omega\chi_{\{\sum_{\ell=1}^n(u_\ell+\eps)> L\}}
	\bigg(1+\sum_{i=1}^nu_i\bigg)dxdt
	\to \int_0^s\int_\Omega\chi_{\{\sum_{\ell=1}^nu_\ell\geq L\}}
	\bigg(1+\sum_{i=1}^nu_i\bigg)dxdt.
\end{align*}

We perform the limit $\eps\to 0$ in the remaining terms. By dominated cnvergence,
we find that
\begin{align*}
  G_{12} &= \sum_{i=1}^n\int_0^s\int_\Omega\varphi^M_K(u+\eps I)\frac{u_i}{u_i+\eps}
	b_i\cdot\na u_idxdt
  \to \sum_{i=1}^n\int_0^s\int_\Omega\chi_{\{u_i>0\}}\varphi^M_K(u)b_i\cdot\na u_idxdt,\\
  I_{62} &= -\sum_{i=1}^n\int_0^s\int_\Omega\varphi^L_K(u+\eps I)b_i\cdot\na u_idxdt \\
  &\to -\sum_{i=1}^n\int_0^s\int_\Omega\varphi^L_K(u)b_i\cdot\na u_idxdt
  = -\sum_{i=1}^n\int_0^s\int_\Omega\chi_{\{u_i>0\}}\varphi^L_K(u)b_i\cdot\na u_idxdt.
\end{align*}

Let us consider the integrand of $G_3$. Using the definition for $A_{ij}$,
we obtain the pointwise convergence as $\eps\to 0$:
\begin{align*}
  \pa_k\varphi^M_K&(u+\eps I)\big(\log (u_i+\eps)+\lambda_i\big)
  \bigg(\sum_{j=1}^nA_{ij}(u)\na u_j-u_ib_i\bigg)\cdot\na u_k \\
  &= 2\pa_k\varphi^M_K(u+\eps I)\sqrt{u_i}\big(\log (u_i+\eps)+\lambda_i\big)
	\bigg(a_{i0}+\sum_{\ell=1}^na_{il}u_\ell\bigg)\na \sqrt{u_i}\cdot\na u_k \\
  &\phantom{xx}{}+ \pa_k\varphi^M_K(u+\eps I)u_i\big(\log (u_i+\eps)+\lambda_i\big)
	\bigg(\sum_{j=1}^na_{ij}\na u_j-b_i\bigg)\cdot\na u_k \\
  &\to 2\pa_k\varphi^M_K(u)\sqrt{u_i}(\log u_i+\lambda_i)
	\bigg(a_{i0}+\sum_{\ell=1}^na_{i\ell}u_\ell\bigg)\na\sqrt{u_i}\cdot\na u_k \\
  &\phantom{xx}{}+ \pa_k\varphi^M_K(u)u_i(\log u_i+\lambda_i)
	\bigg(\sum_{j=1}^na_{ij}\na u_j-b_i\bigg)\cdot\na u_k.
\end{align*}
Taking the modulus and summing over $i=1,\ldots,n$, the left-hand side is bounded
from above by
$$
  C(M,K)\sum_{i,k=1}^n\big(|\na \sqrt{u_i}|+|\na u_i|+1\big)|\na u_k|
	\le C(M,K)\sum_{i=1}^n\big(|\na u_i|^2+|\na \sqrt{u_i}|^2+1\big),
$$
which is an $L^1(\Omega\times(0,T))$ function. Therefore, we can use the
dominated convergence theorem again to infer that
\begin{align*}
  G_3 &\to -2\sum_{i,k=1}^n\int_0^s\int_\Omega\pa_k\varphi^M_K(u)\sqrt{u_i}
	(\log u_i+\lambda_i)\bigg(a_{i0}+\sum_{\ell=1}^na_{i\ell}u_\ell\bigg)
	\na\sqrt{u_i}\cdot\na u_k dxdt \\
  &\phantom{xx}{} -\sum_{i,k=1}^n\int_0^s\int_\Omega\pa_k\varphi^M_K(u)u_i
	(\log u_i+\lambda_i)\bigg(\sum_{j=1}^na_{ij}\na u_j-b_i\bigg)\cdot\na u_kdxdt.
\end{align*}
Similarly, the limit $\eps\to 0$ in $G_4$ gives
$$
  G_4 \to -2\sum_{i,k=1}^n\int_0^s\int_\Omega\pa_i\varphi^M_K(u)\sqrt{u_k}
	(\log u_k+\lambda_k)\bigg(\sum_{j=1}^nA_{ij}(u)\na u_j-u_ib_i\bigg)
	\cdot\na\sqrt{u_k}dxdt.
$$
The limit $\eps\to 0$ in the remaining terms $G_2$, $G_5$, $I_1$, $I_2$, $I_3$,
$I_5$, $I_7$ follows
directly from property (L3) and the dominated convergence theorem.
We conclude from \eqref{2.conc3} and \eqref{2.HKLM} that
\begin{equation}\label{2.conc4}
  H_{K}^{M,L}(u|v)\Big|_0^s + \sum_{i=1}^ne^{-\lambda_i}\int_\Omega
	\varphi^M_K(u)dx\Big|_0^s
	\le P_1+\cdots+P_{15},
\end{equation}
where
\begin{align*}
  P_1 &= -2\eta_0\sum_{i=1}^n\int_0^s\int_\Omega\chi_{\{\sum_{\ell=1}^n u_\ell\geq L\}}
	\varphi^M_K(u)|\na u_i|^2dxdt, \\
	P_2 &= \eta_0\sum_{i=1}^n\int_0^s\int_\Omega\chi_{\{\sum_{\ell=1}^n u_\ell\geq L\}}
	|\na u_i|^2dxdt, \\
  P_3 &= C(L,f,v)\int_0^s\int_\Omega\chi_{\{\sum_{\ell=1}^nu_\ell\leq L\}}
	\sum_{i=1}^n|u_i-v_i|^2dxdt, \\
  P_4 &= C(L,K,f,v)\int_0^s\int_\Omega\chi_{\{\sum_{\ell=1}^nu_\ell\geq L\}}
	\bigg(1+\sum_{i=1}^nu_i\bigg)dxdt, \\
  P_5 &= \sum_{i=1}^n\int_0^s\int_\Omega\chi_{\{u_i>0\}}
	\big(\varphi^M_K(u)-\varphi^L_K(u)\big)b_i\cdot\na u_idxdt, \\
  P_6 &= -\sum_{i,k=1}^n\int_0^s\int_\Omega\pa_i\pa_k\varphi^M_K(u)
	\sum_{\ell=1}^n\big(u_\ell(\log u_\ell+\lambda_\ell-1)+e^{-\lambda_\ell}\big) \\
  &\phantom{xx}{}\times\bigg(\sum_{j=1}^nA_{ij}(u)\na u_j-u_ib_i\bigg)\cdot\na u_kdxdt, \\
  P_7 &= -2\sum_{i,k=1}^n\int_0^s\int_\Omega\pa_k\varphi^M_K(u)\sqrt{u_i}
	(\log u_i+\lambda_i)\bigg(a_{i0}+\sum_{\ell=1}^na_{i\ell}u_\ell\bigg)
	\na \sqrt{u_i}\cdot\na u_kdxdt, \\
  P_8 &= -\sum_{i,k=1}^n\int_0^s\int_\Omega\pa_k\varphi^M_K(u)u_i
	(\log u_i+\lambda_i)\bigg(\sum_{j=1}^na_{ij}\na u_j-b_i\bigg)\cdot\na u_kdxdt, \\
  P_{9} &= -2\sum_{i,k=1}^n\int_0^s\int_\Omega\pa_i\varphi^M_K(u)\sqrt{u_k}
	(\log u_k+\lambda_k)\bigg(\sum_{j=1}^nA_{ij}(u)\na u_j-u_ib_i\bigg)
	\cdot\na \sqrt{u_k}dxdt, \\
  P_{10} &= \sum_{i=1}^n\int_0^s\int_\Omega\pa_i\varphi^M_K(u)\sum_{\ell=1}^n
	\big(u_\ell(\log u_\ell+\lambda_\ell-1)+e^{-\lambda_\ell}\big)f_i(u)dxdt, \\
  P_{11} &= \sum_{i,j=1}^n\int_0^s\int_\Omega \pa_j\varphi^L_K(u)
	(\log v_i+\lambda_i)\bigg(\sum_{\ell=1}^nA_{j\ell}(u)\na u_\ell-u_jb_j\bigg)
	\cdot\na u_idxdt, \\
  P_{12} &= \sum_{i,j=1}^n\int_0^s\int_\Omega \pa_j\varphi^L_K(u)
	(\log v_i+\lambda_i)\bigg(\sum_{\ell=1}^nA_{i\ell}(u)\na u_\ell-u_ib_i\bigg)
	\cdot\na u_jdxdt, \\
  P_{13} &= \sum_{i,j,k=1}^n\int_0^s\int_\Omega u_i\pa_j\pa_k\varphi^L_K(u)
	(\log v_i+\lambda_i)\bigg(\sum_{\ell=1}^nA_{j\ell}(u)\na u_\ell-u_jb_j\bigg)
	\cdot\na u_kdxdt, \\
  P_{14} &= \sum_{i,j=1}^n\int_0^s\int_\Omega u_i\pa_j\varphi^L_K(u)
	\bigg(\sum_{\ell=1}^nA_{j\ell}(u)\na u_\ell-u_jb_j\bigg)\cdot\frac{\na v_i}{v_i}dxdt, \\
  P_{15} &= \sum_{i,j=1}^n\int_0^s\int_\Omega u_i\pa_j\varphi^L_K(u)
	\bigg(\sum_{\ell=1}^nA_{i\ell}(v)\na v_\ell-v_ib_i\bigg)\cdot\frac{\na u_j}{v_i}dxdt.
\end{align*}


\subsection{The limit $M\to \infty$}

We perform the limit $M\to\infty$ in \eqref{2.conc4}. Observe that the terms
$P_2,\ldots,P_4$ and $P_{11},\ldots,P_{15}$ do not depend on $M$ such that
we need to pass to the limit only in the remaining terms. First, we consider
the left-hand side of \eqref{2.conc4}. Recall that
\begin{align*}
  H_\eps^{M,L}(u|v) &+ \sum_{i=1}^n e^{-\lambda_i}\int_\Omega\varphi_K^M(u)dx \\
	&= \int_\Omega\bigg(\varphi_K^M(u)h(u)
	-\sum_{i=1}^n\big(\varphi_K^L(u)u_i(\log v_i+\lambda_i)-v_i\big)\bigg)dx,
\end{align*}
where $h(u)$ is defined in \eqref{1.h}.
Since $|\varphi_K^M(u)h(u)|\le h(u)$ and $\varphi_K^M(u)\to 1$ pointwise a.e.\
as $M\to\infty$, we infer from the dominated convergence theorem that
\begin{align*}
  \int_\Omega \varphi^M_K(u)h(u)dx\Big|_0^s
  \to \int_\Omega h(u)dx\Big|_0^s
	&= \int_\Omega \sum_{\ell=1}^n u_\ell(\log u_\ell+\lambda_\ell-1)dx\Big|_0^s
	+ \sum_{i=1}^ne^{-\lambda_i}\int_\Omega dx\Big|_0^s \\
  &= \int_\Omega \sum_{\ell=1}^n u_\ell(\log u_\ell+\lambda_\ell-1)dx\Big|_0^s.
\end{align*}
This shows that in the limit $M\to\infty$,
\begin{align*}
  & H_{K}^{M,L}(u|v)\Big|_0^s + \sum_{i=1}^n e^{-\lambda_i}\int_\Omega
	\varphi^M_K(u)dx\Big|_0^s \to H_K^L(u|v)\Big|_0^s \\
	&:=	\int_\Omega \bigg(\sum_{\ell=1}^n u_\ell(\log u_\ell+\lambda_\ell-1)
	- \sum_{i=1}^n\big(\varphi^L_K(u)u_i(\log {v_i}+\lambda_i)-v_i\big)
	\bigg)dx\Big|_0^s.
\end{align*}

We turn to the terms on the right-hand side of \eqref{2.conc4}. Clearly,
as $M\to\infty$,
$$
  P_1 \to -2\eta_0\sum_{i=1}^n\int_0^s\int_\Omega\chi_{\{\sum_{\ell=1}^n u_\ell\geq L\}}
	|\na u_i|^2dxdt.
$$
Recall that $P_3$ and $P_4$ do not depend on $M$. Furthermore,
$$
  P_5 \to \sum_{i=1}^n\int_0^s\int_\Omega\chi_{\{u_i>0\}}\big(1-\varphi^L_K(u)\big)
	b_i\cdot\na u_idxdt.
$$
We use (L5) to estimate the following part of the integrand of $P_6$:
\begin{align*}
   \bigg|\pa_i\pa_k\varphi^M_K(u) & \big(u_\ell(\log u_\ell+\lambda_\ell-1)
	+ e^{-\lambda_\ell}\big)(1+u_j)\bigg| \\
  &\le C(K)\frac{[u_\ell(\log u_\ell+\lambda_\ell-1) + e^{-\lambda_\ell}](1+u_j)}{
	(\sum_{i=1}^n u_i + e)^2\log(\sum_{i=1}^n u_i + e)}\le C(K).
\end{align*}
Thus, the integrand of $P_6$ is bounded from above by
$$
  C(K)\sum_{j=1}^n(|\na u_j|+1)|\na u_k|
	\le C(K)\sum_{j=1}^n\big(|\na u_j|^2+1\big)\in L^1(\Omega\times(0,T)).
$$
We deduce from $\pa_i\pa_k\varphi_K^M(u)\to 0$ as $M\to\infty$ that
$$
  P_6\to 0\quad\mbox{as }M\to\infty.
$$

We rewrite the term $P_7$ as
\begin{align*}
  P_7 &= -2\sum_{i,k=1}^n\int_0^s\int_\Omega\chi_{\{u_i\leq 1\}}\pa_k\varphi^M_K(u)
	\sqrt{u_i}(\log u_i+\lambda_i)\bigg(a_{i0}+\sum_{\ell=1}^na_{il}u_\ell\bigg)
	\na \sqrt{u_i}\cdot\na {u_k}dxdt\\
  &\phantom{xx}{} -\sum_{i,k=1}^n\int_0^s\int_\Omega\chi_{\{u_i> 1\}}
	\pa_k\varphi^M_K(u)(\log u_i+\lambda_i)\bigg(a_{i0}+\sum_{\ell=1}^na_{il}u_\ell\bigg)
	\na {u_i}\cdot\na u_k dxdt.
\end{align*}
Since
\begin{align*}
  \bigg|\chi_{\{u_i\leq 1\}}\pa_k\varphi^M_K(u)\sqrt{u_i}(\log u_i+\lambda_i)
	(1+u_j)\bigg|
  &\le \frac{C(K)(1+u_j)}{(\sum_{i=1}^n u_i +e)\log(\sum_{i=1}^n u_i +e)}\le C(K), \\
  \bigg|\chi_{\{u_i>1\}}\pa_{k}\varphi^M_K(u)(\log u_i +\lambda_i)(1+u_j)\bigg|
  &\le \frac{C(K)\chi_{\{u_i> 1\}}(\log u_i +\lambda_i)(1+u_j)}{(\sum_{i=1}^n u_i+e)
	\log(\sum_{i=1}^n u_i+e)}\le C(K),
\end{align*}
the integrand of $P_7$ is bounded from above by
$$
  C(K)\sum_{i,k=1}^n\big(|\na u_i|+|\na\sqrt{u_i}|\big)|\na u_k|
	\le C(K)\sum_{i=1}^n\big(|\na u_i|^2+|\na \sqrt{u_i}|^2\big).
$$
and the right-hand side is a function in $L^1(\Omega\times(0,T))$.
We infer from $\lim_{M\to\infty}\partial_k\varphi^M_K(u)=0$ and Lebesgue's dominated
convergence theorem that $P_7\to 0$ as $M\to\infty$.
Similarly, we infer that
$$
  P_8\to 0, \quad P_9\to 0\quad\mbox{as }M\to\infty.
$$

It remains to estimate $P_{10}$ as $P_{11},\ldots,P_{15}$ do not depend on $M$.
For this, we make explicit the derivative $\pa_i\varphi_K^M(u)$:
\begin{align*}
  P_{10} &= \int_0^s\int_\Omega\chi_{\{\sum_{k=1}^nu_k\geq M\}}
  \varphi'\bigg(\frac{\log\log(\sum_{k=1}^nu_k+e)-\log\log(M+e)}{\log(K+1)}\bigg) \\
  &\phantom{xx}{}\times\frac{\sum_{\ell=1}^n[u_\ell(\log u_\ell+\lambda_\ell-1)
	+ e^{-\lambda_\ell}]}{\log(K+1)({\sum_{k=1}^nu_k}+e)\log({\sum_{k=1}^nu_k}+e)}
  \sum_{i=1}^nf_i(u)dxdt.
\end{align*}
According to condition (H2.iii), there exists $M_0\in\N$
such that for all $\sum_{i=1}^n u_i\ge M_0$,
it holds that $\sum_{i=1}^n f_i(u)\ge 0$ if $M\ge M_0$,
and hence from $\varphi'\leq 0$ that
$P_{10}\le 0$. 

\begin{remark}\label{rem.alt}\rm
If we assume that $|\sum_{i=1}^n f_i(w)|\le C(1+|w|^p)$ for all $w\in[0,\infty)^n$,
we can conclude that $P_{10}\to 0$ as $M\to\infty$. Indeed, 
it follows from the Gagliardo-Nirenberg 
inequality (as shown in \cite[page 732]{CDJ18}) that $u_i\in L^p(\Omega\times(0,T))$ 
with $p=2+2/d$. This implies that $\sum_{i=1}^n f_i(u)\in L^1(\Omega\times(0,T))$, 
and we deduce from
$\lim_{M\to\infty}\chi_{\{\sum_{k=1}^nu_k\geq M\}(x,t)}=0$ and Lebesgue's dominated
convergence theorem that $P_{10}\to 0$ as $M\to\infty$.
\qed
\end{remark}

In conclusion, we obtain from \eqref{2.conc4} in the limit $M\to\infty$,
\begin{equation}\label{2.conc5}
  H_K^L(u|v)\Big|_0^s \le Q_1 + \cdots + Q_4 + P_{11} + \cdots + P_{15},
\end{equation}
where
\begin{align*}
  Q_1 &= -\eta_0\sum_{i=1}^n\int_0^s\int_\Omega\chi_{\{\sum_{\ell=1}^n u_\ell\geq L\}}
	|\na u_i|^2dxdt, \\
  Q_2 &= C(L,f,v)\int_0^s\int_\Omega\chi_{\{\sum_{\ell=1}^nu_\ell\leq L\}}
	\sum_{i=1}^n|u_i-v_i|^2dxdt, \\
  Q_3 &= C(L,K,f,v)\int_0^s\int_\Omega\chi_{\{\sum_{\ell=1}^nu_\ell\geq L\}}
	\bigg(1+\sum_{i=1}^nu_i\bigg)dxdt, \\
  Q_4 &= \sum_{i=1}^n\int_0^s\int_\Omega\chi_{\{u_i>0\}}\big(1-\varphi^L_K(u)\big)
	b_i\cdot\na u_idxdt,
\end{align*}
and we recall that the terms $P_{11},\ldots,P_{15}$ are defined after \eqref{2.conc4}.


\subsection{End of the proof}

We claim that the right-hand side of \eqref{2.conc5} can be bounded from above
by $\int_0^s H_K^L(u|v)dt$ (up to a constant), which then allows for a
Gronwall argument to conclude that $H_K^L(u|v)=0$.
To this end, we estimate the terms $Q_i$ and $P_i$.

The terms $Q_2$ and $Q_3$ can be bounded from above by a constant times the
entropy $H_K^L(u|v)$. This was shown by Fischer in \cite{Fis17}, and we recall
his result for the convenience of the reader.
\begin{lemma}[Lemma 9 in \cite{Fis17}]\label{lem.fis}
There exists $L\in \N$ such that for all $K\in \N$,
\begin{align}
  \int_\Omega\chi_{\{\sum_{\ell=1}^nu_\ell\geq L\}}\bigg(1+\sum_{i=1}^nu_i\bigg)dx
	&\le 2H_{K}^{L}(u|v), \label{2.fis1} \\
  \int_\Omega\chi_{\{\sum_{\ell=1}^nu_\ell\leq L\}}\sum_{i=1}^n|u_i-v_i|^2dx
	&\le C(L)H_{K}^{L}(u|v). \label{2.fis2}
\end{align}
\end{lemma}
Hence, we infer that
$$
  Q_2  +Q_3 \le C(L,K,f,v)\int_0^s H_K^L(u|v)dt.
$$

It follows from (L1), (L2), Young's inequality, and Lemma \ref{lem.fis} that
\begin{align*}
  Q_4 &= \sum_{i=1}^n\int_0^s\int_\Omega\chi_{\{u_i>0\}}
	\chi_{\{\sum_{\ell=1}^nu_\ell\geq L\}}\big(1-\varphi^L_K(u)\big)b_i\cdot\na u_idxdt \\
  &\le C(b)\sum_{i=1}^n\int_0^s\int_\Omega\chi_{\{\sum_{\ell=1}^nu_\ell\geq L\}}
	|\na u_i|dxdt \\
  &\le \frac{\eta_0}{2}\sum_{i=1}^n\int_0^s\int_\Omega
	\chi_{\{\sum_{\ell=1}^nu_\ell\geq L\}}|\na {u_i}|^2dxdt
  + C(b)\int_0^s\int_\Omega\chi_{\{\sum_{\ell=1}^nu_\ell\geq L\}}dxdt \\
  &\le \frac{\eta_0}{2}\sum_{i=1}^n\int_0^s\int_\Omega
	\chi_{\{\sum_{\ell=1}^nu_\ell\geq L\}}|\na {u_i}|^2dxdt
	+ C(b)\int_0^s H_K^L(u|v)dt,
\end{align*}
and the first term on the right-hand side can be absorbed by $Q_1$.
In a similar way, using (L2), (L4), and Lemma \ref{lem.fis}, we have
\begin{align*}
  & P_{11} + P_{12} \\
	&= \sum_{i,j=1}^n\int_0^s\int_\Omega \chi_{\{\sum_{\ell=1}^nu_\ell\geq L\}}
	\pa_j\varphi^L_K(u)(\log v_i+\lambda_i)
	\bigg(\sum_{\ell=1}^nA_{j\ell}(u)\na u_\ell-u_jb_j\bigg)\cdot\na u_idxdt, \\
  &\phantom{xx}{} +\sum_{i,j=1}^n\int_0^s\int_\Omega
	\chi_{\{\sum_{\ell=1}^nu_\ell\geq L\}}\pa_j\varphi^L_K(u)(\log v_i+\lambda_i)
	\bigg(\sum_{\ell=1}^nA_{i\ell}(u)\na u_\ell-u_ib_i\bigg)\cdot\na u_jdxdt \\
  &\le \frac{C(v,b)}{\log(K+1)}\sum_{i,j=1}^n\int_0^s\int_\Omega
	\chi_{\{\sum_{\ell=1}^nu_\ell\geq L\}}|\na u_i|(|\na u_j|+1)dxdt \\
  &\le \frac{C(v,b)}{\log(K+1)}\sum_{i=1}^n\int_0^s\int_\Omega
	\chi_{\{\sum_{\ell=1}^nu_\ell\geq L\}}|\na u_i|^2dxdt
	+ C(v,b)\int_0^s H_K^L(u|v)dt.
\end{align*}
Furthermore, taking into account (L2), (L5), and Lemma \ref{lem.fis},
\begin{align*}
  P_{13} &= \sum_{i,j,k=1}^n\int_0^s\int_\Omega\chi_{\{\sum_{\ell=1}^nu_\ell\geq L\}}
	u_i\pa_j\pa_k\varphi^L_K(u)(\log v_i+\lambda_i) \\
	&\phantom{xx}{}\times
	\bigg(\sum_{\ell=1}^nA_{j\ell}(u)\na u_\ell-u_jb_j\bigg)\cdot\na u_kdxdt \\
  &\le \frac{C(v,b)}{\log(K+1)}\sum_{i,j=1}^n\int_0^s\int_\Omega
	\chi_{\{\sum_{\ell=1}^nu_\ell\geq L\}}|\na u_i|(|\na u_j|+1)dxdt \\
  &\le \frac{C(v,b)}{\log(K+1)}\sum_{i=1}^n\int_0^s\int_\Omega
	\chi_{\{\sum_{\ell=1}^nu_\ell\geq L\}}|\na u_i|^2dxdt + C(v,b)\int_0^s H_K^L(u|v)dt.
\end{align*}
Finally, using (L2)-(L4) and estimating as before:
\begin{align*}
  P_{14} + P_{15}
	&\leq \frac{C(v,b)}{\log(K+1)}\sum_{i,j=1}^n\int_0^s\int_\Omega
	\chi_{\{(L+e)^{K+1}>\sum_{\ell=1}^nu_\ell\geq L\}}
	\big(u_i|\na u_j|+u_i+|\na u_j|\big)dxdt \\
  &\leq \frac{C(v,b)}{\log(K+1)}\sum_{i=1}^n\int_0^s\int_\Omega
	\chi_{\{(L+e)^{K+1}>\sum_{\ell=1}^nu_\ell\geq L\}}|\na u_i|^2dxdt \\
  &\phantom{xx}{} +\frac{C(v,b)}{\log(K+1)}\int_0^s\int_\Omega
	\chi_{\{(L+e)^{K+1}>\sum_{\ell=1}^nu_\ell\geq L\}}
	\bigg(1+\sum_{i=1}^nu_i^2\bigg)dxdt \\
  &\leq \frac{C(v,b)}{\log(K+1)}\sum_{i=1}^n\int_0^s\int_\Omega
	\chi_{\{\sum_{\ell=1}^nu_\ell\geq L\}}|\na u_i|^2dxdt
	+ C(L,K,v,b)\int_0^s H_K^L(u|v)dt.
\end{align*}

Summarizing, we infer from \eqref{2.conc5} that
\begin{align*}
  H_K^L(u|v)\Big|_0^s &\le C(L,K,f,v,b)\int_0^s H_K^L(u|v)dt \\
	&\phantom{xx}{}
	+ \bigg(-\frac{\eta_0}{2}+\frac{C(v,b)}{\log(K+1)}\bigg)\sum_{i=1}^n\int_0^s
	\int_\Omega\chi_{\{\sum_{\ell=1}^nu_\ell\geq L\}}|\na u_i|^2dxdt.
\end{align*}
Choosing $K\in\N$ sufficiently large, the second term on the
right-hand side is nonpositive and consequently,
$$
  H_K^L(u|v)\Big|_0^s \le C(L,K,f,v)\int_0^s H_K^L(u|v)dt.
$$
It remains to determine $L\in\N$. Since we assumed that the initial data $u^0$ is
bounded, we choose $L\in\N$ such that $\sum_{i=1}^n u_i^0<L$. Then
$\varphi_K^L(u^0)=1$ and
$H_K^L(u(0)|v(0))=H_K^L(u^0,u^0)=0$. The Gronwall lemma shows that
$H_K^L(u(s)|v(s))=0$ for all $s\in(0,T^*)$. We claim that this yields
$u(s)=v(s)$ for $s\in(0,T^*)$. Indeed, by \eqref{2.fis2} in Lemma \ref{lem.fis},
it follows that $u_i(s)=v_i(s)$ in $\{\sum_{\ell=1}^n u_\ell\le L\}$
for all $i=1,\ldots,n$ and $s\in(0,T^*)$. Furthermore, by \eqref{2.fis1} in
Lemma \ref{lem.fis}, we have $\mbox{meas}(\{\sum_{\ell=1}^n u_\ell\ge L\})=0$.
Therefore, $u_i(s)=v_i(s)$ on $\Omega$, which concludes the proof.


\end{document}